\documentclass[11pt]{amsart}
\usepackage{amssymb, amscd}
\usepackage{latexsym,epsfig}
\usepackage{extarrows}

\usepackage[
        hypertexnames=false,
        hyperindex,
        pagebackref,
        breaklinks=true,
        bookmarks=false,
        colorlinks,
        linkcolor=blue,
        citecolor=red,
        urlcolor=red,
]{hyperref}
\usepackage[all]{xy}
\usepackage{pst-all}
\makeatletter
\def\legendre@dash#1#2{\hb@xt@#1{%
  \kern-#2\p@
  \cleaders\hbox{\kern.5\p@
    \vrule\@height.2\p@\@depth.2\p@\@width\p@
    \kern.5\p@}\hfil
  \kern-#2\p@
  }}
\def\@legendre#1#2#3#4#5{\mathopen{}\left(
  \sbox\z@{$\genfrac{}{}{0pt}{#1}{#3#4}{#3#5}$}%
  \dimen@=\wd\z@
  \kern-\p@\vcenter{\box0}\kern-\dimen@\vcenter{\legendre@dash\dimen@{#2}}\kern-\p@
  \right)\mathclose{}}
\newcommand\legendre[2]{\mathchoice
  {\@legendre{0}{1}{}{#1}{#2}}
  {\@legendre{1}{.5}{\vphantom{1}}{#1}{#2}}
  {\@legendre{2}{0}{\vphantom{1}}{#1}{#2}}
  {\@legendre{3}{0}{\vphantom{1}}{#1}{#2}}
}
\def\dlegendre{\@legendre{0}{1}{}}
\def\tlegendre{\@legendre{1}{0.5}{\vphantom{1}}}
\makeatother

\numberwithin{equation}{section}
\def\today{\ifcase\month\or Jan\or Febr\or  Mar\or  Apr\or May\or Jun\or  Jul\or Aug\or  Sep\or  Oct\or Nov\or  Dec\or\fi \space\number\day, \number\year}

\newcommand{\EE}{\mathbb E}
\newcommand{\FF}{\mathbb F}

\newcommand{\HH}{\mathbb H}

\newcommand{\PP}{\mathbb P}
\newcommand{\QQ}{\mathbb Q}

\newcommand{\ZZ}{\mathbb Z}

\newcommand{\A}{\mathcal A}

\newcommand\langepijl[1]{\buildrel {#1} \over \longrightarrow}

\newtheorem{theorem}{Theorem}[section]
\newtheorem{lemma}[theorem]{Lemma}
\newtheorem{proposition}[theorem]{Proposition}
\newtheorem{corollary}[theorem]{Corollary}

\newtheorem{definition-lemma}[theorem]{Definition-Lemma}

\theoremstyle{definition}
\newtheorem{definition}[theorem]{Definition}

\theoremstyle{remark}
\newtheorem{remark}[theorem]{Remark}

\begin{document}

\title[]{The Cycle Class of the Supersingular Locus\\ 
of principally polarized abelian varieties}
\author{Gerard van der Geer}
\address{Korteweg-de Vries Instituut, Universiteit van
Amsterdam, Postbus 94248,
1090 GE  Amsterdam, The Netherlands.}
\email{g.b.m.vandergeer@uva.nl}

\author{Shushi Harashita}
\address{Graduate School of Environment and Information Sciences, Yokohama National University.}
\email{harasita@ynu.ac.jp}
\begin{abstract}
We prove a formula for the cycle class of the          
supersingular locus in the Chow ring with rational coefficients
of the moduli space of principally polarized abelian varieties of
dimension $g$ in characteristic $p$.
This formula determines this class as a monomial in the
Chern classes of the Hodge bundle
up to a factor that is a polynomial in $p$.
This factor is known for $g\leq 3$. We also determine the factor for $g=4$.
\end{abstract}

\maketitle
\begin{section}{Introduction}
An abelian variety over a field $k$ of characteristic $p>0$ is called supersingular
if it is isogenous to a product of supersingular elliptic curves over
the algebraic closure of $k$. Equivalently, by \cite[Thm.\ 4.2]{Oort1974}, 
if its formal isogeny type
has a Newton polygon with all slopes equal to $1/2$.  Recall that the
Newton
polygon of an abelian variety starts at $(0,0)$, ends at $(2g,g)$ and
is lower convex and satisfies a symmetry condition. The two extreme cases
are the Newton polygon with slopes $0$ and $1$ and break point $(g,0)$ (the ordinary
case) and the Newton polygon with slope $1/2$ (the supersingular case).

Let ${\A}_g \otimes {\FF}_p$ be the moduli space of principally polarized
abelian varieties of dimension $g$ in characteristic $p>0$.
The supersingular locus $S_g$ is defined as the closed subset of principally
polarized abelian varieties of ${\A}_g\otimes {\FF}_p$ that are supersingular.
This locus can be considered as the most degenerate stratum in the Newton
polygon stratification on ${\A}_g \otimes {\FF}_p$. Its dimension is known
by Li and Oort to be $[g^2/4]$ 
and also the number of irreducible components is
known, see below. Besides the Newton polygon stratification there is another
stratification on $\mathcal{A}_g\otimes {\FF}_p$, the Ekedahl-Oort 
stratification. While the cycle classes of the Ekedahl-Oort stratification
on ${\A}_g \otimes {\FF}_p$ are known, the cycle classes of the Newton polgon
strata in general are not. For $g=1$ and $g=2$ the supersingular locus
is a (closed) stratum of the Ekedahl-Oort stratification and thus the class is known.
For $g=3$ the supersingular locus is not a (closed) 
stratum of the Ekedahl-Oort stratification, 
but its cycle class was determined in joint work of the first
author with Ekedahl, and the result was presented in \cite{vdG1}.

In this paper we will prove a formula for the cycle class of the 
supersingular locus in the Chow ring with rational coefficients
of a Faltings-Chai
compactification $\tilde{\mathcal{A}}_g\otimes {\FF}_p$.
This formula determines this class as a monomial in the
Chern classes of the Hodge bundle 
up to a factor that is a polynomial in $p$. 
This shows that this class lies in the tautological ring, a subring of the
rational Chow ring  of the moduli space
$\tilde{\mathcal{A}}_g \otimes {\FF}_p$, 
and is given by a beautiful formula
that generalizes Deuring's famous formula for the number of supersingular elliptic curves.

\begin{theorem} \label{MainThm}
The cycle class of the supersingular locus $S_g$ in the Chow ring
with rational coefficients of a Faltings-Chai
compactification $\tilde{\mathcal{A}}_{g} \otimes {\FF}_p$ of the moduli space
$\mathcal{A}_g \otimes {\FF}_p$ lies
in the tautological ring. More precisely, it is of the form
$$
[S_g]=f_g(p) \begin{cases} \lambda_g \lambda_{g-2} \cdots \lambda_2 & \text{$g$ even}, \\
\lambda_g \lambda_{g-2} \cdots \lambda_1 & \text{$g$ odd}\, , \\
\end{cases}
$$
where $f_g(p)$ is a polynomial in $p$ with rational coefficients
and $\lambda_i$ is the $i$th Chern class of the Hodge bundle on
$\tilde{\mathcal{A}}_{g} \otimes {\FF}_p$.
\end{theorem}

The method for proving this rests upon a way to translate conditions on the
supersingularity of Dieudonn\'e modules into degeneracy conditions of
morphisms of bundles made from the Hodge bundle.
We hope that this method can be used to show that all the cycle classes
of the Newton polygon stratification lie in the tautological ring. 
In principle, the method should lead
to the determination of the missing factor, but the details of nailing this down
seem formidable. 

We  also determine the factor $f_g(p)$ in the formula for the
cycle class of the supersingular 
locus for $g=4$ by different methods. This determination 
builds upon the method used for the case of $g=3$
and calculates the degree of a Chern class
of the Hodge bundle on a component of the supersingular locus.
For this we construct an explicit smooth model of each
irreducible component of $S_4$.
In addition, for completeness we give the proof for the class for $g=3$
that was not published in \cite{vdG1}. 
Including the well-known results for $g=1$ and $g=2$ we arrive at
the following theorem.

\begin{theorem} \label{2ndThm}
The cycle class of the supersingular locus $S_g$ in the Chow ring 
with rational coefficients of a Faltings-Chai
compactification $\tilde{\mathcal{A}}_{g} \otimes {\FF}_p$ of the moduli space
$\mathcal{A}_g \otimes {\FF}_p$ for $g\leq 4$ is given by 
$$
[S_g]= \begin{cases}  
(p-1) \, \lambda_1 & $g=1$ \\
(p-1)(p^2-1)\, \lambda_2 & $g=2$ \\
(p-1)^2(p^3-1)(p^4-1) \, \lambda_3\lambda_1 & $g=3$ \\
(p-1)^3(p^3-1)(p^4-1)(p^6-1)\, \lambda_4\lambda_2 & $g=4$.\\
 \end{cases}
$$
\end{theorem}

We also discuss for $g=3$ and $g=4$ 
the loci in the supersingular locus where the
$a$-number is at least $2$.

{\sl Acknowledgement}:
Research was supported by JSPS Grant-in-Aid for Scientific Research (C) 21K03159.
The first author thanks YMSC of Tsing\-hua University in Beijing and MPIM in Bonn 
for support. We thank the referee for his/her remarks.
\end{section}
\eject
\tableofcontents
\begin{section}{The moduli space ${\A}_g \otimes {\FF}_p$}\label{Space-Ag}
By ${\A}_g$ we denote 
the moduli stack of principally polarized abelian
varieties of dimension $g$  and by $\pi: \mathcal{X}_g \to {\A}_g$
the universal abelian variety over ${\A}_g$. It is a Deligne-Mumford stack 
defined over ${\ZZ}$. The moduli space
${\A}_g$ carries a natural vector bundle ${\EE}$ of rank $g$,
the Hodge bundle, defined
as $\pi_{*} \Omega^1_{\mathcal{X}_g/{\A}_g}$. We denote by $\tilde{\A}_g$
a Faltings-Chai compactification of ${\A}_g$ as defined and treated
in \cite{F-C}. The Hodge bundle extends to
$\tilde{\A}_g$ and will again be denoted by ${\EE}$.

In the rest of this paper we consider the moduli stack $\mathcal{A}_g \otimes {\FF}_p$
in characteristic $p>0$. 
Let ${\rm CH}_{\QQ}^*(\tilde{\A}_g\otimes {\FF}_p)$ be the Chow ring
with rational coefficients of $\tilde{\A}_g\otimes {\FF}_p$.
We set
$\lambda_i=c_i({\EE})\in {\rm CH}_{\QQ}^i(\tilde{\A}_g\otimes {\FF}_p)$ 
for the $i$th Chern class  of ${\EE}$ for
$i=1,\ldots,g$, see \cite[Ch.\ 3]{Fulton}. These classes satisfy the relation
$$
(1+\lambda_1+\cdots +\lambda_g)(1-\lambda_1+\cdots +(-1)^g\lambda_g)=1
$$
and these classes generate a subring $R_g$ of the Chow ring
${\rm CH}_{\QQ}^*(\tilde{\A}_g\otimes {\FF}_p)$ called the tautological ring,
see \cite{vdG1,E-V}. 
For $0\leq n \leq g(g+1)/2$
the graded part of $R_g$ of degree $n$ has a basis
$\lambda_1^{e_1} \cdots \lambda_g^{e_g}$ with $0\leq e_i \leq 1$ and
$\sum_i  e_i i=n$.
The ring $R_g$ is a Gorenstein ring with socle generated by
$\lambda_1\lambda_2 \cdots \lambda_g$. 
We will denote the degree of this $0$-cycle by
$$
v(g)= \deg \lambda_1\lambda_2 \cdots \lambda_g\, ,
$$
the Hirzebruch proportionality constant, 
and we have
$$
v(g)=
(-1)^{g(g+1)/2} 2^{-g} \zeta(-1)\zeta(-3) \cdots \zeta(1-2g),
$$
where $\zeta(s)$ is the Riemann zeta function, see \cite{vdG1}. We give a little table
with relevant values:

\vbox{
\bigskip\centerline{\def\quad{\hskip 0.6em\relax}
\def\quod{\hskip 0.5em\relax }
\vbox{\offinterlineskip
\hrule
\halign{&\vrule#&\strut\quod\hfil#\quad\cr
height2pt&\omit&&\omit&&\omit&&\omit&&\omit&&\omit&\cr
& $g$ && $0$  && $1$ && $2$ && $3$ && $4$ &\cr
\noalign{\hrule}
& $v(g)$ && $1$  && $1/24$ && $1/5760$  && $1/2903040$ && $1/1393459200$& \cr
} \hrule}
}}

\bigskip

The tautological ring of $\mathcal{A}_g \otimes {\FF}_p$ is the quotient
$R_g/(\lambda_g)\cong R_{g-1}$.

\vskip0.5truecm
The moduli space ${\A}_g\otimes {\FF}_p$ carries two 
important stratifications, the Eke\-dahl-Oort stratification and
the Newton polygon stratification, see 
\cite{Oort:Stratification} and \cite{Oort-AM152}. The strata of the Ekedahl-Oort stratification
$\mathcal{V}_{\mu}$ are indexed
by Young diagrams or tuples $\mu=[\mu_1,\ldots,\mu_r]$ of integers with
$0\leq r \leq g$ and $\mu_i> \mu_{i+1}$, according to the
usage of \cite{vdG1,E-vdG}. The largest open stratum
$\mathcal{V}_{[\emptyset]}$ is the locus of ordinary abelian varieties.
The codimension of $\mathcal{V}_{\mu}$ is $\sum_i \mu_i$.
The stratification can be extended to $\tilde{\A}_g$.

By \cite{vdG1,E-vdG} we can calculate 
the cycle classes of the closed Ekedahl-Oort strata in $\mathcal{A}_g \otimes {\FF}_p$
and $\tilde{\A}_g \otimes {\FF}_p$. For example the cycle class of the locus of
abelian varieties with $p$-rank $\leq f$
(corresponding to $\mu=[g-f]$)
is
$$
[\overline{\mathcal{V}}_{[g-f]}]=(p-1)(p^2-1)\cdots (p^{g-f}-1) \lambda_{g-f} \eqno(1)
$$
and the cycle class of the smallest stratum, the locus of superspecial
abelian varieties (corresponding to $\mu=[g,g-1,\ldots,1]$) is
$$
[\mathcal{V}_{[g,g-1,\ldots,1]}]=(p-1)(p^2+1)\cdots (p^g+(-1)^g) \lambda_1\lambda_2 \cdots \lambda_g\, .
$$
This formula implies as a special case a result of Ekedahl \cite{Ekedahl1987},  namely that
$$
\sum \frac{1}{\# {\rm Aut}(X)}= (p-1)(p^2+1)\cdots (p^g+(-1)^g) \, v(g)\, ,
\eqno(2)
$$
where the sum is over the isomorphism classes of principally polarized superspecial 
abelian varieties over $\overline{\FF}_{p}$ and $v(g)$ the proportionality constant
defined above. A formula for the actual number of
isomorphism classes of superspecial abelian varieties with a level $n\geq 3$
structure is obtained
by multiplying the formula for the degree of
$\mathcal{V}_{[g,g-1,\ldots,1]}$ by the degree of the natural map
${\A}_g[n] \to {\A}_g$ (as stacks)
with ${\A}_g[n]$ the moduli space of principally polarized abelian varieties
with a level $n$ structure.
\end{section}
\begin{section}{Irreducible components of the supersingular locus}
The number of irreducible components of the supersingular locus $S_g$ 
in $\mathcal{A}_g \otimes {\FF}_p$ was determined
by Deuring for $g=1$, by Katsura and Oort for $g=2$ (\cite{K-O,K-O2}) 
and in general by 
Li and Oort for $g\geq 3$, \cite[4.9]{L-O}. 
The actual number of irreducible components in
${\A}_g\otimes \overline{\FF}_p$  
is given by a class number $h_p(g)$ for $g$ odd and a
similar class number $h_p'(g)$ for $g$ even.
Here $h_p(g)$ (resp. $h_p'(g)$) is the class number of the principal
(resp.\ non-principal) genus in the hermitian space $B^g$, with $B$
the definite quaternion algebra over ${\QQ}$ ramified exactly at $p$ and $\infty$.
These class numbers are difficult
to deal with, see for example \cite[p.\ 147]{I-K-O}, 
and one gets better and more useful formulas by counting 
in a stacky way, that is, taking into account weights equal to the inverse
of the order of the automorphism groups of the objects that one counts.
For example, for $g=1$ the class number of the quaternion algebra
$B$
over ${\QQ}$ split outside $p$ and $\infty$, is given by
$$
h_p(1)= \frac{p-1}{12}+ \left( 1-\legendre{-3}{p}\right)\frac{1}{3}+
\left( 1- \legendre{-4}{p}\right) \frac{1}{4}\, ,
$$
with the Legendre symbols.
But a stacky interpretation of this number leads to 
the much more elegant formula
$$
\sum \frac{1}{\# {\rm Aut}(E)} = \frac{p-1}{24}
$$
with the summation over all isomorphism classes of
supersingular elliptic curves defined over
$\overline{\FF}_p$.

We will denote by $N_g$ the number of irreducible components of the
supersingular locus in the stacky sense, that is, where each irreducible component is counted with a 
certain weight $w$ related to the automorphism group as explained below.

This number $N_g$ has the property that the 
number $N_g[n]$ of irreducible components of the
supersingular locus on the moduli space ${\A}_g[n]$ with a level
$n \geq 3$ structure with $p$ prime to $n$ equals 
$$
N_g[n]=N_g \cdot
\deg({\A}_g[n]\to {\A}_g)\, .
$$
An irreducible component of the supersingular locus of
$\mathcal{A}_g[n]$ is given by a triple $(E^g, \eta, \nu)$
with $E$ a supersingular elliptic curve, $\eta$ a polarization with
kernel equal to the kernel $E^g[F^{g-1}]$ of $F^{g-1}$ with $F$ Frobenius 
and $\nu$ a level $n$ structure, see \cite{L-O} and the next section. 
Since
$p$ does not divide $n$, a level $n$ structure on $E^g$ does not interfere
with the inseparable isogenies $E^g \to Y_0$ that give rise to the objects of
an irreducible component and descends to a level $n$ structure on $Y_0$.
So we count such an  irreducible component of the supersingular locus of $\mathcal{A}_g$
with weight $w=1/\# {\rm Aut}(E^g,\eta)$. 
\begin{proposition}\label{Ng}
The number $N_g$
of irreducible components of the supersingular locus in 
$\mathcal{A}_g \otimes \overline{\FF}_p$ (in the stacky sense) is
$$
\begin{cases}
(p-1)(p^2+1)(p^3-1)\cdots (p^g-1) \, v(g) & \text{for $g$ odd}, \\
(p^2-1)(p^6-1)\cdots (p^{2g-2}-1)\, v(g) & \text{for $g$ even}. \\
\end{cases}
$$
\end{proposition}
The stacky interpretation that we use reduces to the mass of
the principal (resp. non-principal genus) and can be deduced from
\cite{Ekedahl1987} or \cite{H-I}.
One finds this mass formula also in  \cite[p.\ 123]{Gan-Hanke-Yu} and
in \cite{Yu}.

For odd $g$ the irreducible components of $S_g$ are in
bijective correspondence with the isomorphism classes of superspecial
principally polarized abelian varieties of dimension $g$, hence
the formula for $N_g$ follows immediately from Ekedahl's result (2).
For even $g$ one has a correction factor 
$$
\frac{(p+1)(p^3+1)\cdots (p^{g-1}+1)}{(p^2+1)(p^4+1)\cdots (p^{g}+1)}\, .
$$
Here for $g$ even the numerator can be interpreted as the number of 
totally isotropic subspaces of dimension $g/2$ 
in a $g$-dimensional
hermitian space over ${\FF}_{p^2}$ with conjugation given by Frobenius,
while the denominator equals the 
number of totally isotropic subspaces of dimension $g/2$
in a symplectic space of dimension $g$ over ${\FF}_{p^2}$. See also the
description in 
\cite{Yu}. 
\end{section}
\begin{section}{Flag type quotients}\label{FTQ}
Work of Oda and Oort (\cite{Oda-Oort}) 
makes it possible to parametrize the irreducible 
components of the supersingular locus $S_g$ by so-called flag type 
quotients. For an abelian variety $X$ over an algebraically closed field $k$ of
characteristic $p$ we will denote the 
subgroup scheme $\ker(F)\cap \ker(V)$
by $A(X)$ with $F$ and $V$ Frobenius and Verschiebung on $X$. 
It is a subgroup scheme of order $p^{a(X)}$ with $a(X)$ the
$a$-number of $X$. A supersingular abelian variety has $1 \leq a(X) \leq g$
and if $a(X)=g$ and $g\geq 2$ then $X$ is isomorphic to the base change to $k$
of a product $E^g$ with $E$ a supersingular elliptic curve defined over
${\FF}_p$. 

For a supersingular abelian variety $X$ of dimension $g$
the $a$-number tends to go up when one replaces $X$ by $X/A(X)$, though
it is not true that $a(X/A(X))\geq \min(g,a(X)+1)$ as asserted in the proof of \cite[1.8 Lemma]{L-O} 
that refers to \cite{K-Z-L};
see \cite[Remark 3.17]{KYY} for a counterexample. 
Nevertheless, by starting with $X=X_0$ 
and putting $X_{i+1}=X_i/A(X_i)$ one arrives after  $g-1$
steps at a superspecial abelian variety $X_{g-1}$, that is, 
an abelian variety with 
with $a(X_{g-1})=g$, as follows from \cite[Lemma 9]{Zink}. Then the kernel of the dual map is 
contained in $\ker(F^{g-1})$, hence one finds a homomorphism $Y\to X$
with $Y=X_{g-1}^{(p^{g-1})}$. 
 This implies the fact that for
a supersingular abelian variety $X$ there exists a minimal isogeny
$\rho: E^g \to X$ with $E$ a supersingular elliptic curve with the
property that any other homomorphism $h: Z \to X$ of a superspecial 
abelian variety $Z$ factors uniquely through $\rho$. If $a(X)=1$ 
this minimal isogeny is obtained in $g-1$ steps
$$
Y_{g-1} \to Y_{g-2} \to \cdots \to Y_0=X
$$
where $Y_{g-1}=E^g\otimes {\rm Spec}(k)$ and $Y_i=Y_{g-1}/G_i$ for
$i=1,\ldots,g-1$ with $G_i=\ker(\rho)\cap Y_{g-1}[F^{g-1-i}]$.
If $a(X)>1$ this sequence needs not be unique. Taking into account also
the polarizations leads to the definition of a (polarized) flag type quotient. 

\begin{definition} \label{pftq}
A polarized flag type quotient of dimension $g$ 
is a diagram of abelian varieties and
homomorphisms
\[
\begin{xy}
\xymatrix{
Y_{g-1} \ar[r]^{\rho_{g-1}} \ar[d]^{\eta_{g-1}} & Y_{g-2} \ar[r] \ar[d]^{\eta_{g-2}} &\cdots
\ar[r]^{\rho_1} & Y_0 \ar[d]^{\eta_0} \\
Y^t_{g-1}& \ar[l] Y^t_{g-2} & \ar[l] \cdots & \ar[l] Y_0^t \\
}
\end{xy} 
\]
where $Y_i^t$ is the dual of $Y_i$ and the abelian variety $Y_{g-1}$ is superspecial with 
$\eta_{g-1}$ a polarization with kernel $Y_{g-1}[F^{g-1}]$ satisfying
\begin{enumerate}
\item{} $\ker(\rho_i)\subset A(Y_i)$ is of order $p^i$;
\item{} $\ker(\eta_i)\subseteq \ker(V^j \circ F^{i-j})$ for $0 \leq j \leq i/2$\, .
\end{enumerate}
This flag type quotient is called rigid if $G_i=G_0\cap G[F^{g-1-i}]$
with $G_0=\ker (Y_{g-1}\to Y_0)\cap Y_{g-1}[F^{g-1}]$.
The term `rigid' refers to the fact that in this case the corresponding flag type is unique.
\end{definition}

\bigskip

The main references for flag type quotients are \cite{Oda-Oort}
and \cite[Sections 7,9.6,9.7]{L-O}.
\end{section}
\begin{section}{Dieudonn\'e modules and displays}\label{DMD}
The theory of Dieudonn\'e modules makes it possible to describe flag type quotients
in terms of Dieudonn\'e modules. 

Here $k$ will denote an algebraically closed
field of characteristic $p$ and $W=W(k)$ the ring of Witt vectors of $k$.
We define a ring
$$
A=W[F,V]/(FV-p, VF-p, Fa-a^{\sigma}F, aV-Va^{\sigma},  \, \forall a\in W)
$$
and set $A_{1,1}:=A/(F-V)$. 

A polarized flag type quotient as described in Definition \ref{pftq}
 corresponds 
to a flag of contravariant Dieudonn\'e modules
$$
M_0 \subset M_1 \subset M_2 \subset \cdots \subset M_{g-1}
$$
with dual modules $M_i^t$ satisfying
\begin{enumerate}
\item{} $M_{g-1}=A_{1,1}^g$ provided 
with a quasi-polarization 
$$
\langle \, , \, \rangle:M_{g-1} \otimes_W M_{g-1}^t
\to Q(W)\, ,
$$ 
with $Q(W)$ the field of quotients of $W$,
that induces an identification $M_{g-1}^t= F^{g-1}M_{g-1}$;
\item{} $(F,V)M_i \subset M_{i-1}$ and $\dim (M_i/M_{i-1})=i$ for $i=0,\ldots,g-1$;
\item{} $(F,V)^iM_i\subset M_i^t$ for $i=0,\ldots,g-1$.
\end{enumerate}
We call such a flag a polarized Dieudonn\'e flag of length $g$.
It is called 
rigid if $M_i=M_0+F^{g-1-i}M_{g-1}$ for $i=0,\ldots,g-1$.
We observe that rigidity implies
$$
M_i=M_m+F^{g-1-i}M_{g-1} \quad \text{\rm for $m<i\leq (g-1)$}\, . 
$$

We can translate rigid polarized flag type quotients 
in terms of displays, replacing Dieudonn\'e modules by displays.
We recall the definition of displays (cf.~\cite[Section 1]{Display}).
Let $R$ be a commutative unitary ring of characteristic~$p$.
Let $W(R)$ be the ring of Witt vectors.
Let $\mathfrak{f}: W(R) \to W(R)$ be Frobenius and  $\mathfrak{v}: W(R) \to W(R)$ 
Verschiebung.
Set $I_R = \mathfrak{v}(W(R))$.
A {\it display over $R$} is a quadruple $(P,Q,F,V^{-1})$
consisting of a finitely generated projective $W(R)$-module $P$,
a $W(R)$-submodule $Q$ of $P$ and homomorphisms $F: P^{(p)} \to P$
and $V^{-1}: Q^{(p)} \to P$, where $M^{(p)} := W(R)\otimes_{\mathfrak{f},W(R)} M$, with the properties:
\begin{enumerate}
    \item[\rm (i)] $I_RP \subset Q \subset P$ and
    there exists a decomposition of $P$ into a direct sum of $W(R)$-modules $P=L\oplus T$,
     such that $Q=L\oplus I_R T$;
    \item[\rm (ii)] $V^{-1}$ is an epimorphism;
    \item[\rm(iii)] For $x\in P$ and $w\in W(R)$ we have $V^{-1}(1\otimes \mathfrak{v}(w) x) = wFx$.
\end{enumerate}
By \cite[Lemma 9]{Display}, we have an isomorphism
\[
V^{-1}\oplus F: (L \oplus T)^{(p)} \to P. \eqno(3)
\]
The matrix (with respect to a basis of $P$) associated to this isomorphism
is a generalization of the classical display (\cite{Norman}).

\begin{remark}
If $R$ is a perfect field, then $P$ is the usual Dieudonn\'e module,
$I_R=p\, W(R)$ and $Q$ is the $V$-image $VP$, so $Q$ is determined by the Dieudonn\'e module $P$.
But if $R$ is not a perfect field, then $Q$ is not determined by $P$ together 
with $F,V$;
conversely $P$ is determined by the $V^{-1}$-image of $Q^{(p)}$.
\end{remark}

By a result of Li-Oort \cite[3.7]{L-O} the moduli space of polarized Dieudonn\'e flags
of length $g$ exists and is projective. Moreover, by \cite[3.7]{L-O} 
the moduli of rigid 
polarized Dieudonn\'e flags of length $g$ exists and is quasi-projective, and
by \cite[7.6]{L-O} it is non-singular.

\end{section}
\begin{section}{The cycle class of the supersingular locus}\label{class-section}
In this section we will show that the 
cycle class of the supersingular 
locus $S_g$ in $\mathcal{A}_g\otimes {\FF}_p$ lies in the
tautological ring $R_g$ generated by the Chern classes $\lambda_i$ 
($i=1,\ldots,g$) of the Hodge
bundle ${\EE}$ 
on a Faltings-Chai compactification of $\mathcal{A}_g\otimes {\FF}_p$ and give a formula for it that fixes the class
up to a multiplicative constant.  

Here the cycle class is taken
in the Chow ring with rational coefficients 
of a Faltings-Chai compactification 
$\tilde{\mathcal{A}}_g \otimes {\FF}_p$ of $\mathcal{A}_g \otimes {\FF}_p$.

\begin{theorem}\label{NPlociThm}
The cycle class of the supersingular locus on
$\mathcal{A}_g\otimes {\FF}_p$ in ${\rm CH}_{\QQ}^{*}(\tilde{\mathcal{A}}_g\otimes {\FF}_p)$
is a non-zero
multiple of $\lambda_g \lambda_{g-2} \cdots \lambda_{1}$ if $g$ is odd
and of  $\lambda_g \lambda_{g-2} \cdots \lambda_{2}$ if $g$ is even. 
The multiple is a polynomial in $p$ with rational coefficients.
\end{theorem}

Before we give the details of the proof we describe the set-up.
For the proof we will use the presentation of Frobenius on the
covariant Dieudonn\'e module $M$ of $p$-rank $0$ and $a$-number $1$
with a principal quasi-polarization $\langle \ , \ \rangle$ 
as given by Oort in \cite{Oort-AM152}.
His description of the display of such a module $M$ is as follows.
With $W$ the Witt ring of $k$, an algebraically closed field of
characteristic $p>0$,  
there exists a $W$-basis $e_1,\ldots,e_g,e_{g+1},\ldots,e_{2g}$ that is symplectic
(meaning that 
$\langle e_i,e_j \rangle = 0$ for $i,j\leq g$
and for $i,j > g$ and
$\langle e_i,e_{g+j} \rangle = \delta_{ij}$ for $1\leq i,j \leq g$) 
such that Frobenius is given by the formulas
$$
\begin{aligned}
&Fe_j=\sum_{i=1}^{2g} \gamma_{ij} e_i, \quad (1\leq j \leq g)\, , \\
&e_{j}=V(\sum_{i=1}^{2g} \gamma_{ij}e_i)\, ,
\quad  (g+1\leq  j \leq 2g)\, ,\\
\end{aligned}
$$
where $\gamma=(\gamma_{ij})$
is a $W$-valued $2g\times 2g$ matrix which is symplectic in the sense that
$$
\gamma \left( \begin{matrix} 0 & 1_g \\ -1_g & 0 \\ \end{matrix} \right)
\gamma^t= \left( \begin{matrix} 0 & 1_g \\ -1_g & 0 \\ \end{matrix} \right)\, .
$$
We write $\gamma$ as a matrix of $g \times g$ blocks 
$$
\gamma=(\gamma_{ij})=\left(\begin{matrix} a & b \\ c & d \\ \end{matrix}\right)\, .
$$
We denote the Frobenius endomorphism of the Witt ring $W$ by $\sigma$.
Note that the $\sigma$-linear map $F$ is given by the matrix
$$
\left( \begin{matrix} a & pb \\ c & pd \\ \end{matrix} \right)\, .
$$
Oort shows (\cite[p.\ 191]{Oort-AM152} that if $M$ has $p$-rank $0$ and $a$-number $1$ 
we may choose the basis
such that the matrix $\gamma$ is of the form (called normal)
$$
a_{ij}=d_{ij}=\begin{cases} 1 & i=j+1, \\ 0 & i\neq j+1, \\ \end{cases}, \quad
c_{ij}=\begin{cases} 1 & (i,j)=(1,g) \\ 0 & else \\ \end{cases}, \\
$$
and $b_{ig}=0$ for $i\neq 1$. In particular, since we assume $p$-rank zero 
we have $a_{ig}=0$
for $i=1,\ldots,g$, see \cite[page 191]{Oort-AM152} after Lemma 2.2.

\begin{lemma} \label{gamma-matrix} For a normal form $\gamma$ we have 
$\gamma_{i,2g}=0$ for $i=2,\ldots,g$ and $\gamma_{1,j}=0$ for $j=g+1,\ldots,2g-1$
and $\gamma_{1,2g}=-1$. Moreover,
the square matrix 
$$
\tilde{\gamma}=\left( \begin{matrix} \gamma_{2,g+1} & \ldots & \gamma_{2,2g-1} \\
\vdots && \vdots \\
\gamma_{g,g+1} & \cdots & \gamma_{g,2g-1} \\ \end{matrix} \right)
$$
is symmetric.
\end{lemma}
\begin{proof} We have $ab^t=ba^t$ and $b^td=d^tb$. In view of the shape of the
matrices $a$ and $d$ the result follows as $\gamma$ is symplectic.
\end{proof}

We now change this normal form into a so-called strong normal form as follows.
We can take $\gamma_{ij}$ as a Teichm\"uller lift for $i\ne g$ and $j\ne 2g-1$, 
after changing the basis $\{e_i\}$ of $M$. Now we consider the $\gamma_{i,j}$ only for $2\leq i < g$ 
and $g < j < 2g-1$, as the others are kept as Teichm\"uller lifts under the following operation.  
Let $t_{ij}:=(\gamma_{i,j}-[\overline{\gamma_{i,j}}])/p$, where $[u]$ denotes the Teichm\"uller 
lift of $u\in k$. We replace $e_{j+1}$  by $e_{j+1} + p\, t_{ij}e_i$ and $e_{g+i}$
by  $e_{g+i} + p \, t_{ij}e_{j+1-g}$. 
After the change this new basis is still symplectic and
the new $\gamma_{i,j}$ becomes the Teichm\"uller lift $[\overline{\gamma_{ij}}]$ and
the new $\gamma_{i+1,j+1}$ becomes $\gamma_{i+1,j+1} + t_{ij}^\sigma$;
by symmetry (Lemma \ref{gamma-matrix}), similar things hold for $\gamma_{j+1-g, g+i-1}$
and $\gamma_{j+2-g, g+i}$;
at the same time the other new $\gamma_{i'j'}$ do not change.

By carrying out this operation going from lower $i+j$ to higher, we get the desired the basis.
We call such $(\gamma_{ij})$ {\it a strong normal form}.

Given such a basis in strong normal form, we have according to 
\cite[Lemma 2.6]{Oort-AM152} that there exists an element $P\in A$ such that
$$
F^{2g}e_1=Pe_1 \quad \text{\rm with} \quad
P=\sum_{i=1}^g \sum_{j=g}^{2g} p^{j-g} \gamma_{ij}^{{\sigma}^{2g-j}} F^{2g+i-j-1} \, ,
\eqno(4)
$$
with $Fx=x^{\sigma}F$ for $x\in W$ and repeated application
of $F$ is in the $\sigma$-linear sense (cf.\ \cite[p.\ 195]{Oort-AM152}).

\begin{remark}
We know that the Ekedahl-Oort stratum $\mathcal{V}_{\mu}$ with $\mu=[g,1]$
corresponding to $p$-rank $0$ and $a$-number $2$ has codimension $1$ in 
the $p$-rank~$0$ locus $V_0$,
hence the generic point of every irreducible component of $V_0$ has $a=1$.
Moreover, by the results of Li-Oort \cite{L-O} we know that each irreducible
component of the supersingular locus $S_g$ has an open dense subset where 
the $a$-number equals $1$. 
\end{remark}
One can read off supersingularity from the matrix $\tilde{\gamma}$ in strong normal form using
Oort's result on the action of $F$ on $e_1$ given in (4), 
see \cite[Cor.\ 2.8]{Oort-AM152}.

\begin{corollary}\label{ss-criterion} Let $\gamma$ be the matrix in strong normal form
for the module $M$. 
Then the module $M$ is supersingular if 
$\gamma_{ij}\equiv 0 \, (\bmod \, p)$ 
for $2 \leq i \leq g-1$,
$g+1 \leq j \leq 2g-2$ with $i+j\leq 2g$.
Equivalently, since $\gamma_{ij}$ is a Teichm\"uller lift, if $\gamma_{ij}=0$ 
for $2 \leq i \leq g-1$,
$g+1 \leq j \leq 2g-2$ with $i+j\leq 2g$.
\end{corollary}
Note that because of the symmetry this gives a priori
$$
\sum_{j=1}^{\lfloor g/2 \rfloor} (g-2j) 
=\frac{g(g-1)}{2}- [\frac{g^2}{4}]= \dim V_0 - \dim S_g
$$ 
conditions for supersingularity, 
where $V_0$ is the $p$-rank zero locus. 

We now begin the proof of Theorem \ref{NPlociThm}.
\begin{proof}
The strategy is now to impose consecutively conditions that together imply
supersingularity by Corollary \ref{ss-criterion}, where 
we assume that $\gamma$ is in strong normal form; we begin by requiring the vanishing
modulo $p$ of the column of entries that is the transpose of
$$
(\gamma_{2,g+1},\ldots,\gamma_{g-1,g+1})\, ,
$$
and continue by requiring the vanishing modulo $p$ of the column of entries
whose transpose is
$$
(\gamma_{3,g+2},\ldots,\gamma_{g-2,g+2})\, ,
$$
and so on, 
till finally 
the column with transpose $(\gamma_{g/2,3g/2-1},\gamma_{g/2+1,3g/2-1})$ of length $2$
for $g$ even or the vanishing of the single entry
$\gamma_{(g+1)/2,(3g-1)/2}$ for $g$ odd.

For example, for $g=5$ we require the vanishing modulo $p$ of the red entries in the symmetric matrix
$$
\tilde{\gamma} =\left(
\begin{matrix}
\color{red}\gamma_{26} & \color{red}\gamma_{27} &\color{red}\gamma_{28}& \gamma_{29} \\
\color{red}\gamma_{36} & \color{red}\gamma_{37} & \gamma_{38} & \gamma_{39}\\
\color{red}\gamma_{46} & \gamma_{47} & \gamma_{48} & \gamma_{49} \\
\gamma_{56} & \gamma_{57} & \gamma_{58} & \gamma_{59} \\
\end{matrix} \right)
$$
giving $4$ conditions.

In terms of displays, 
we have an $\mathfrak f$-linear map $V^{-1}\oplus F: M=L\oplus T \overset{}\to M$, see (3). 
We write $F/p$ for the composition 
$$
VM/pM \to M/pM \xrightarrow{V^{-1}\oplus F} M/pM \to M/VM\, .
$$
This map is given by the square matrix $(\gamma_{ij})_{1\leq i \leq g, g+1 \leq j \leq 2g}$.
Then by the vanishing indicated in Lemma \ref{gamma-matrix} we may restrict
to submodules of rank $g-1$ generated by $g-1$ consecutive generators in 
$VM/pM$ and  $M/VM$:
$$
G=\langle e_{g+1},e_{g+2},\ldots,e_{2g-1}\rangle \longrightarrow
H=\langle e_2,e_3,\ldots,e_{g}\rangle \, .
$$
We have increasing 
filtrations for $i=1,\ldots,g-1$ of $G$ and $H$ given by
$$
G_i=\langle e_{g+1},e_{g+2},\ldots,e_{g+i}\rangle \quad \text{\rm and} \quad 
H_i=\langle e_2,e_{3},\ldots,e_{i+1}\rangle\, .
$$
That the $p$-rank is zero means that the image of $G_{g-1}$
is in $H_{g-1}$. If we identify ${\rm Lie}(X)$
with $VM/pM$ for the abelian variety $X$ corresponding to the dual of $M$ (cf.~\cite[4.3.12]{BGM} and \cite[5.4, 7.4]{L-O}),
we can view the induced map $F/p: G_{g-1} \to H_{g-1}$
as a symmetric
morphism between vector bundles of rank $g-1$ 
made from the Hodge bundle and
its dual by Frobenius twists. 
Since we wish to have the filtrations we will have
to work on a cover of the $p$-rank zero locus $V_0$.

We now consider $G\langle 1 \rangle$, the module generated by $e_{g+1}$.
We require that it maps to zero modulo $p$ under $F/p: G\langle 1 \rangle \to H\langle 1\rangle $ 
with the module $H\langle i \rangle$ generated by  $e_{i+1},\ldots,e_{g-i}$.
We can view the semi-linear map $G\langle 1 \rangle \to H\langle 1 \rangle$ defined by $F/p$ modulo $p$
 as a morphism
of a line bundle to a vector bundle of rank $g-2$, where these bundles are made from 
the Hodge bundle by truncations and Frobenius twists. 
We consider the locus where this morphism vanishes. 
The vanishing of this morphism corresponds to the vanishing
modulo $p$ of the vector $(\gamma_{2,g+1},\ldots,\gamma_{g-1,g+1})$.

If this morphism vanishes
then by the symmetry $\gamma_{2,g+2}$ vanishes modulo $p$ and we can consider a morphism
$G\langle 2 \rangle \to H\langle 2 \rangle$ induced by $F/p$ with $G\langle j \rangle=G_j/G_{j-1}$ 
generated by $e_{g+j}$ and require its vanishing modulo $p$.
By induction, assuming the vanishing modulo $p$ of the semi-linear
morphism 
$$
G\langle j \rangle \longrightarrow  H\langle j \rangle \eqno(5)
$$
for $j=1,\ldots,s$, 
we get a next morphism $G\langle s+1\rangle \to H\langle s+1 \rangle$.
We require inductively that these morphisms vanish for $j=1,\ldots,[(g-1)/2]$
on an appropriate covering space of $V_0$ where we have the filtrations.
Supersingularity follows if
the conditions that the induced map
$ G\langle j \rangle \to H\langle j \rangle $
is zero are satisfied successively for $j=1,\ldots, [(g-1)/2]$.

The locus where the morphism (5) vanishes has cycle class expressed in the Chern
classes of $G\langle j\rangle$ and $H\langle j \rangle$; for example for $j=1$ 
the cycle class 
is the $(g-2)$th Chern class of the dual of $G\langle 1\rangle \otimes
(H\langle 1 \rangle)^{\vee}$. 

\bigskip

We now work on the space of flags $\mathfrak{F}$ 
on the cohomology $H^1_{\rm dR}$
of the universal principally polarized abelian variety
as introduced in \cite[Section 3]{E-vdG}. The de Rham cohomology sheaf ${\mathcal H}^1_{\rm dR}(X/S)$
for a principally abelian variety $X \to S$ is a locally free sheaf ${\HH}$ 
of rank $2g$ on $S$ fitting in an
exact sequence
$$
0 \to {\EE} \to {\HH} \to {\EE}^{\vee} \to 0 \, .
$$
The flags in question
are complete symplectic flags on ${\HH}$ extending flags ${\EE}(i)$ on the Hodge bundle
with ${\rm rank}({\EE}(i))=i$
for $i=1,\ldots,g$. These 
flags on the de Rham cohomology sheaf ${\HH}$ satisfy ${\EE}(g+i)={\EE}(g-i)^{\bot}$ and thus are
determined by the flag on ${\EE}$.
This flag space is a stratified space with strata indexed by elements of the Weyl group
of the symplectic group. 
The stratum corresponding to the longest so-called final element (or Kostant element) 
of the Weyl group 
(see \cite[Section 3]{E-vdG}) 
parametrizes flags compatible with the action of $V$ and $F$. Its closure contains
the final stratum lying over the $p$-rank zero locus $V_0$.

Thus we work on the closure of the 
final stratum $\mathfrak{F}_w$ of $\mathfrak{F}$ corresponding to 
$p$-rank zero. This stratum allows a morphism that is generically finite to $V_0$.
The symplectic flags over a generic point of $V_0$ are compatible with the action
of $V$ and $F$ and also compatible with the filtration defined by the 
basis used in the description by Oort of the
display given above.  

We can view the induced map $F/p: G_{g-1} \to H_{g-1}$ 
as a symmetric 
morphism between modules of rank $g-1$ that induces a morphism of vector bundles $G\langle 1 \rangle \to H\langle 1 \rangle$ on $\mathfrak{F}_w$.
The vector bundles induced by $G$ and $H$ have filtrations
whose graded quotients are Frobenius twists of 
of the graded quotients
of the Hodge bundle
${\EE}(i)/{\EE}(i-1)$ or their duals.
Therefore the Chern classes of their graded quotients are of the form
$\pm p^{r_i} \ell_i$ where $\ell_i=c_1({\EE}(i)/{\EE}(i-1))$ ($i=1,\ldots,g$) 
are the Chern classes of the graded quotients 
of the Hodge bundle on the final stratum and $r_i \in {\ZZ}$.

The conditions on the vanishing modulo $p$ of rows of entries
can now be viewed as a degeneracy condition for a 
morphism between vector
bundles on $\mathfrak{F}_w$. 

We shall calculate the cycle class of the Zariski closure
of the degeneracy locus of this map over the open part of $V_0$
where $a=1$. Note that on the open stratum $\mathcal{F}_w$ we have
$a=1$.
This Zariski closure is contained in the supersingular locus
as the Newton polygon can only go up under specialization. Moreover,
for $g\geq 2$
each irreducible component of $S_g$ has an open dense set with
$a=1$, hence intersects the degeneracy locus over $V_0$.

We know that the codimension of the degeneracy locus equals the
number of conditions imposed by Corollary \ref{ss-criterion}
in the supersingular case, hence also for the intermediate cases
defined by the vanishing of $G\langle j \rangle \to H\langle j \rangle$.
The theory of degeneracy loci \cite{F-P1689} 
tells us that the cycle classes of these degeneracy
loci on $\mathfrak{F}_w$  are  polynomials in the classes $\ell_i$. 

To calculate these, we begin by remarking that the cycle class
of the $p$-rank zero locus $V_0$ in 
$\tilde{\mathcal{A}}_g \otimes {\FF}_p$
is a multiple of $\lambda_g$ by \cite{E-vdG}.
We carry out induction and assume that the image under the Gysin map from $\mathfrak{F}_w$ to $\mathcal{A}_g \otimes {\FF}_p$ of the class of the locus over $V_0$ 
where $F/p$ maps $G\langle s\rangle $ to zero in $H\langle s \rangle$
for $s=1,\ldots,j-1$ is a multiple of
$\lambda_g \lambda_{g-2} \cdots \lambda_{g+2-2j}$.  

The locus where the morphism 
$G\langle j\rangle \to H\langle j \rangle$ is zero has 
as cycle class  the $(g-2j)$th Chern class of the dual of
$G\langle j\rangle \otimes (H\langle j \rangle)^{\vee}$.
With $r=g-2j={\rm rank}(H\langle j \rangle)$ 
this Chern class is 
$$
(-1)^r(c_r(H\langle j\rangle) - 
c_{r-1}(H\langle j\rangle) c_1(G\langle j\rangle)) \, .
$$
In order to calculate the class of the corresponding locus
on $\tilde{\mathcal{A}}_g\otimes {\FF}_p$  we have to apply
a Gysin map from the Chow group of $\mathfrak{F}_w$ to the Chow group
of $\tilde{\mathcal{A}}_g \otimes {\FF}_p$  and calculate
the image of the class of the degeneracy locus. 

We first look at the case $j=1$.
\begin{lemma}\label{pushdown}
The pushdowns 
to $\tilde{\mathcal{A}}_g\otimes {\FF}_p$ of the classes 
$c_{g-2}(H\langle 1\rangle)$ and
$c_{g-3}(H\langle 1 \rangle) c_1(G\langle 1\rangle)$ 
on $\mathfrak{F}_{w}$
are multiples of $\lambda_{g-2}$.
\end{lemma}
\begin{proof}
The filtration on ${\EE}$ is extended to the de Rham bundle by 
${\EE}_{g+i}=({\EE}_{g-i})^{\bot}$ as in \cite[Section 3]{E-vdG}.
This symplectic pairing is different from the one used in the description of
the display in \cite{Oort-AM152}.
Since we use covariant Dieudonn\'e modules we have to take duals and Frobenius twists
to relate the Chern roots of $G\langle j\rangle$ and $H\langle j \rangle$ to those
of the Hodge bundle. 
The Chern roots of $G\langle j\rangle$ and $H\langle j \rangle$ 
are determined by the filtrations $G_i$ and $H_i$. We write $l_i$ for these roots,
while writing $\ell_i$ for the roots of ${\EE}$.
Then the Chern roots of  $H\langle 1 \rangle $ given by this filtration
are $l_{2},...,l_{g-1}$ and that of $G\langle 1\rangle$ is $-l_1$.
The Chern class $c_{g-2}(H\langle 1 \rangle)$ 
 is then the $(g-2)$th elementary symmetric function in $l_2,\ldots,l_{g-1}$. 
The $(g-2)$th symmetric function in $l_2,\ldots, l_{g-1}$ is a Frobenius twist
of the $(g-2)$th symmetric function in $l_1,\ldots, l_{g-2}$ (cf.\
the proof of \cite[Lemma 12.3]{E-vdG}) and 
is a multiple of $\lambda_{g-2}(g-2)=c_{g-2}({\EE}(g-2))$. Now by 
\cite[Lemma 12.3]{E-vdG} the pushdown of 
$\lambda_{g-2}(g-2)$ equals a non-zero multiple
of $\lambda_{g-2}$. The morphism from $\mathfrak{F}$ to $\mathcal{A}_g \otimes {\FF}_p$
is fibered by generically finite morphisms $\pi_i$ 
defined by forgetting a step of the flag ${\EE}(i)\subsetneq \EE(i+1)\subsetneq 
\cdots \subsetneq {\EE}(g)$. 
We have for the Chern classes
$\lambda_r(i)=c_r({\EE}(i))$ of the partial flag
the formula
$(\pi_i)^{*}(\lambda_r(i+1))=\ell_{i+1}\lambda_{r-1}(i)+\lambda_r(i)$.
For the Chern roots $l_i$ that we use here a similar formula holds.
Therefore, again by \cite[Lemma 12.3]{E-vdG}, 
the pushdown of $c_{g-3}(H\langle 1 \rangle) c_1(G\langle 1\rangle)$ is also a multiple of $\lambda_{g-2}$.
\end{proof}
We conclude that the class of the 
locus where $G\langle 1 \rangle \to H\langle 1 \rangle$ 
vanishes on $V_0$ is a multiple of the class $\lambda_{g-2}$ on $V_0$.
Since this is a Chern class of a vector bundle on $\tilde{\mathcal{A}}_g \otimes {\FF}_p$ and the class of $V_0$ is a multiple of $\lambda_g$ we find
that the class of the vanishing locus in 
$\tilde{\mathcal{A}}_g \otimes {\FF}_p$
of this bundle morphism on $V_0$ is a multiple 
of $\lambda_g\lambda_{g-2}$.

We now carry out induction. We restrict to the locus $Z$ where
 the consecutive morphisms
$G\langle s\rangle \to H\langle s \rangle$ for $s=1,\ldots,j-1$ vanish.
Then 
the class of the locus of 
vanishing of (5) equals up to a sign the $(g-2j)$th Chern class
of $G\langle j \rangle \otimes (H\langle j \rangle)^{\vee}$
and this 
is $c_{g-2j}(H\langle j \rangle)-c_{g-2j-1}(H\langle j \rangle)
c_1(G\langle j \rangle)$. 
By the argument given in Lemma \ref{pushdown} the class
$c_{g-2j}(H\langle j \rangle)$ is a non-zero multiple of 
the $(g-2j)$th elementary symmetric function in $g-2j$ consecutive classes 
$\ell_i$. 
We can view this as obtained by applying a Frobenius power to
$\ell_1,\ldots,\ell_{g-2j}$, or use \cite[Lemma 12.3]{E-vdG}, 
hence this elementary 
symmetric function represents a
multiple of $\lambda_{g-2j}(g-2j)$. The pushdown of this is
a multiple of $\lambda_{g-2j}$. The argument for 
$c_{g-2j-1}(H\langle j \rangle)
c_1(G\langle j \rangle)$ is similar, as in Lemma \ref{pushdown}.
Therefore the cycle class of the vanishing locus is a
multiple of $\lambda_{g-2j}$ on the locus  $Z$ and $Z$ has as class a multiple
of $\lambda_g\lambda_{g-2}\cdots \lambda_{2g+2-2j}$. As  $\lambda_{g-2j}$ 
is the Chern class of a vector bundle on $\tilde{\mathcal{A}}_g \otimes {\FF}_p$
we find as cycle class on $\tilde{\mathcal{A}}_g \otimes {\FF}_p$ a multiple
of $\lambda_g\lambda_{g-2}\cdots \lambda_{2g-2j}$.

By induction we may assume that the class has as coefficient a polynomial in $p$
with rational coefficients 
as this is true for the class of $V_0$.
By the formula for the Chern class of $G\langle j\rangle \otimes (H\langle j \rangle)^{\vee}$
and the fact that under the Gysin map no denominators are introduced we see that
the coefficient is a polynomial in $p$. This finishes the proof of Theorem \ref{NPlociThm}.

\end{proof}

\begin{remark}
i) By analyzing more precisely the characteristic classes of the degeneracy loci 
in the proof, it should be possible
to determine the multiple $f(p)$ as a polynomial in $p$, but this
involves many subtleties. ii)
By interpreting Newton polygon strata contained in the $p$-rank
zero locus as degeneracy loci as 
done in the proof of Theorem \ref{NPlociThm} we saw that the cycle classes
of these loci lie in the tautological ring. This suggests that all Newton polygon classes are tautological.
\end{remark}
\end{section}
\begin{section}{Moduli of flag type quotients for $g=3$}
In this section and the next we calculate the cycle class of the supersingular locus $S_3$.
We consider an irreducible component of 
the space of polarized flags of Dieudonn\'e
modules for $g=3$, defined by the choice of a quasi-polarization on
$A_{1,1}^3$. This space is the Zariski closure 
of the moduli of rigid polarized Dieudonn\'e flags.
A description
was given in \cite[p.\ 58]{L-O}. Thus we look at polarized flags
$(E^3,\eta)=(Y_2,\eta) \langepijl{\rho_2} (Y_1,\eta_1)
\langepijl{\rho_1} (Y_0,\eta_0)$ corresponding to a polarized flag
of Dieudonn\'e modules
$$
M_0 \subset M_1 \subset M_2=A_{1,1}^3=A\langle x,y,z\rangle
$$
with the quasi-polarization given by 
$$
\langle x,Fx\rangle= \langle y,Fy\rangle=\langle z,Fz\rangle=1/p
\, .
$$
Since $FM_2 \subset M_1$ with $\dim(M_1/FM_2)=1$ 
the module $M_1$ is determined by a $1$-dimensional subspace of $M_2/FM_2$,
say generated by a vector $v=ax+by+cz$. The condition 
$(F,V)M_1 \subset M_1^t$ requires $\langle v, Fv\rangle \in W$, that is,
if we view the coefficients
$a,b,c$ as elements of $k$, the condition $(F,V)M_1 \subset M_1^t$ 
is satisfied if and only if 
$$
a^{p+1}+b^{p+1}+c^{p+1}=0 \, .
$$
Thus the moduli space $\mathcal{F}_1$ 
of truncated flags $M_1\subset M_2$ can be identified with a Fermat curve 
$\mathcal{X}_{p+1} \subset  {\PP}^2={\rm Gr}(1,3)$ (when using Dieudonn\'e modules).
The module $M_0$ is
determined by a $2$-dimensional subspace $M_0/FM_1 \subset M_1/FM_1$.
Assuming rigidity,  
we see that it is spanned by two vectors
$$
w_1=v_0, \qquad w_2= \alpha Fx +\beta Fy+\gamma Fz\, ,
$$
and the condition $M_0\subseteq M_0^t$ gives $a\alpha+b\beta+c \gamma=0$.
This implies that $M_1/M_0$ defines a sheaf isomorphic to $\mathcal{O}_{\mathcal{F}_1}(1)$.
Moreover, the degree $p^2$ homomorphism 
$$
\eta_1: Y_1 \to Y_0 \xrightarrow{\sim}  Y_0^t \to Y_1^t 
$$
shows that $M_1/M_1^t$ is self dual, and it defines a locally
free sheaf 
isomorphic to $\mathcal{O}_{\mathcal{F}_1}(1)
\oplus  \mathcal{O}_{\mathcal{F}_1}(-1)$.

This implies that the moduli space of rigid polarized Dieudonn\'e flags
with given quasi-polarization $\eta$ admits a structure
$$
\mathcal{F}_0^0 \to \mathcal{F}_1 \to \mathcal{F}_2=\text{\rm point}
$$
with $\mathcal{F}_0^0$ the open dense part of the ${\PP}^1$-bundle
$\mathcal{F}_0={\PP}(\mathcal{O}_{\mathcal{F}_1}(1)\oplus
\mathcal{O}_{\mathcal{F}_1}(-1))$ 
that is the complement of the
unique section with negative self-intersection number. 
The Zariski closure is obtained by taking the full ${\PP}^1$-bundle
$\mathcal{F}_0$. 

The morphism $\mathcal{F}_0 \to S_3 \subset \mathcal{A}_3\otimes {\FF}_p$ is of 
finite degree onto its image, and the image forms an irreducible component of $S_3$.
The degree equals $\# {\rm Aut}(E^3,\eta)/\{\pm 1\}$, but we may consider instead of
$\mathcal{F}_0$ the stack
by dividing $\mathcal{F}_0$ 
through the action of ${\rm Aut}(E^3,\eta)$ and then have a morphism of degree $1$.
This is what we shall do.
The natural morphism to $\mathcal{A}_3 \otimes {\FF}_p$ contracts the section.
\end{section}
\begin{section}{The cycle class of $S_3$}
Here we give the proof of the formula for the cycle class of $S_3$
stated in \cite[Thm. 11.3]{vdG1}. The first author learned from Ekedahl at that time
how to calculate the Hodge bundle for flag type quotients. Ekedahl
employed this in \cite[Cor. 3.4]{Ekedahl1985}. This idea was used 
in \cite{vdG1} to calculate the cycle class of $S_3$. As done 
at the time of \cite{vdG1}, here we will not use the results of Section \ref{class-section}.

The Chow rings with rational coefficients
of ${\A}_3\otimes {\FF}_p$ and $\tilde{\A}_3\otimes {\FF}_p$ 
are known by \cite{vdG2}. The ring
${\rm CH}_{\QQ}^*(\tilde{\A}_3 \otimes {\FF}_p)$ is generated by the Chern classes
of the Hodge bundle and boundary classes $\sigma_1$ and $\sigma_2$.
A priori the class of $S_3$  is a linear
combination of the generators of ${\rm CH}_{\QQ}^4({\tilde{\A}_3\otimes {\FF}_p})$, viz.\
$\lambda_1^4$, $\lambda_1^3\sigma_1$,
$\lambda_1^2\sigma_1^2$ and $\lambda_1\sigma_1\sigma_2$, see \cite{vdG2}. But since
$S_3\cdot \sigma_1^2=0=S_3\cdot \sigma_2$ we see from the multiplication
table 3f in \cite[p.\ 765]{vdG2} that the class of
$S_3$ is a multiple of $\lambda_1^4=8\, \lambda_1\lambda_3$.
Alternatively, this follows from the fact that $S_3$ 
is contained in $V_0$,
the $p$-rank $0$ locus, whose class is a multiple of $\lambda_3$.

\begin{theorem}
The class of the supersingular locus for genus $3$  in the Chow ring
with rational coefficients of a Faltings-Chai compactification
of $\mathcal{A}_3\otimes {\FF}_p$ is given by
$$
[S_3]= (p-1)^2(p^3-1)(p^4-1) \, \lambda_1\lambda_3\, .
$$
\end{theorem}
\begin{proof}
The class  $[S_3]$  is a multiple of $\lambda_1\lambda_3$ and the multiple
can be determined by calculating the intersection number with $\lambda_2$.
Using the flag type quotients we see above that an irreducible component
of the supersingular locus $S_3$ in $\mathcal{A}_3\otimes \overline{\FF}_p$ 
is the image of a surface
$\mathcal{F}_0$ under a map $\mathcal{F}_0 \to \mathcal{A}_3 \otimes
\overline{\FF}_p$ 
of degree $\# {\rm Aut}(E^3,\eta)/\{ \pm 1 \}$  (or degree $1$ if we consider the corresponding stack)
and $\mathcal{F}_0$ is of the form
$$
\mathcal{F}_0 \langepijl{\pi_0}\mathcal{F}_1 \langepijl{\pi_1} \mathcal{F}_2={\rm point}\, ,
$$
where $\mathcal{F}_i$ parametrizes partial flag type quotients 
$Y_{2} \to \cdots \to Y_i$. More precisely, a component of $S_3$
is the image under a morphism of a ${\PP}^1$-bundle
$\mathcal{B}=\mathcal{F}_0$
over the Fermat curve $\mathcal{F}_1=\mathcal{X}_{p+1}$
of degree $p+1$ in ${\PP}^2$
that blows down the unique section $S$ 
with negative self-intersection number of the ${\PP}^1$-bundle
${\PP}(\mathcal{O}(1)\oplus \mathcal{O}(-1))$  over $\mathcal{X}_{p+1}$.
A point of $\mathcal{F}_1$ corresponds to the choice of a subgroup scheme $\alpha_p^2$
in $E^3[F]$.

If we use contravariant Dieudonn\'e modules 
over a geometric point of ${\mathcal F}_i$
we have for $i=0$ and $i=1$
 an exact squence
$$
0 \to pM_{i+1}/pM_i \to VM_i/pM_i \to VM_{i+1}/pM_{i+1} \to VM_{i+1}/VM_i \to 0.
$$
Over ${\mathcal F}_i$,
we can identify ${\rm Lie}(Y_i)^{\vee}$ with $VM_i/pM_i$ (cf.~\cite[4.3.12]{BGM} and \cite[5.4, 7.4]{L-O}),
more precisely with ${\rm Q}_i/I_{\mathcal{O}_{\mathcal{F}_i}}{\rm P}_i$, where $({\rm P}_i,{\rm Q}_i,F,V^{-1})$ 
is the display associated to $Y_i$. 
(Note that ${\rm Q}_i$ and $I_{\mathcal{O}_{\mathcal{F}_i}} {\rm P}_i$ become $VM_i$ and $pM_i$ 
respectively if we pull them back to the spectrum of a perfect field.)
By the exact sequence we have in the Grothendieck group
$K_0(\mathcal{F}_i)$ the relation
$$
{\rm Lie}(Y_i)^{\vee} ={\rm Lie}(Y_{i+1})^{\vee}-Q_i+ Q_{i}^{(p)}
$$
with $Q_i$ the locally free $\mathcal{O}_{\mathcal{F}_1}$-module defined
by $VM_{i+1}/VM_i$.
Here ${\rm Lie}(Y_{i+1})$ denotes the pull back under $\pi_i$.
We pull back the relation
$$
{\rm Lie}(Y_1)^{\vee}={\rm Lie}(Y_2)^{\vee}-Q_1+Q_1^{(p)}
$$
under $\pi_0$ to $K_0(\mathcal{F}_0)$ and then find
in $K_0(\mathcal{F}_0)$ suppressing the $\pi_0^*$
$$
{\EE}={\rm Lie}(Y_0)^{\vee}= [3]-Q_1+Q_1^{(p)}-Q_0+Q_0^{(p)}\, ,
$$
where the $[3]$ stands for the class of the
trivial rank $3$ bundle $\pi_0^*\pi_1^*({\rm Lie}(Y_2))^{\vee}$.
From the short exact sequence
$$
0 \to VM_1/pM_1 \to VM_2/pM_2 \to VM_2/VM_1 \to 0
$$
we get the exact sequence of vector bundles
$$
0 \to U_1 \to \pi_1^*({\rm Lie}(Y_2)^{\vee}) \to Q_1\to 0
$$
with ${\rm rank}(Q_1)=2$
that comes from the universal tautological exact sequence of bundles
on the Grassmannian. Here $U_1$ has rank $1$ and $\pi_1^*{\rm Lie}(Y_2)$ is trivial. This
implies that $[Q_1]=[3]-[U_1]$ in the Grothendieck group of vector bundles and so the total Chern
class of ${\rm Lie}(Y_0)^{\vee}$ is given by
$$
(1-\ell_1)(1-p\ell_1)^{-1}(1+\ell_0)^{-1}(1+p\ell_0)\, ,
$$
where $\ell_i=c_1(Q_i)$.
Now $\ell_1$ lives on the curve $\mathcal{F}_1=\mathcal{X}_{p+1}$, so $\ell_1^2=0$.
This gives for the classes $\lambda_1$ and $\lambda_2$
the relations in ${\rm CH}_{\QQ}^{*}(\mathcal{F}_0)$
$$
\lambda_1=(p-1)(\ell_0+\ell_1), \quad
\lambda_2=(p-1)^2\ell_0\ell_1- (p-1)\ell_0^2\, .
$$
The identity $\lambda_1^2=2\lambda_2$ that holds in 
the tautological ring $R_3$ 
implies that $(p-1)^2(\ell_0^2-\ell_1^2)=0$, hence $\ell_0^2=0$.
Since $\deg(\ell_1)=p+1$ on $\mathcal{F}_1$ and $\ell_0$ represents $\mathcal{O}(1)$
on the fibres of $\mathcal{F}_0\to \mathcal{F}_1$ we find $\deg(\ell_0\ell_1)=
p+1$. We thus find that $\deg(\lambda_2)=(p+1)(p-1)^2$ on each irreducible
component of $S_3$. We get
$$
\begin{aligned}
\deg (\lambda_2 [S_3])& =(p+1)(p-1)^2\, N_3 \\
&= (p+1)(p-1)^2 (p-1)(p^2+1)(p^3-1)\, v(3)\, .\\
\end{aligned}
$$
On the other hand, $\deg(\lambda_1 \lambda_2\lambda_3)=v(3)$ 
and this implies the result.
\end{proof}

The morphism $\pi_0: \mathcal{F}_0 \to \mathcal{F}_1$ is a ${\PP}^1$-bundle over a Fermat curve
of degree $p+1$ with a section with image $S$. 
The Picard group of $\mathcal{F}_0$ is generated by the pullback under $\pi_0$ 
of the Picard group of $\mathcal{F}_1$ and by the
class of the section $S$.

\begin{proposition}
We have $[S]=\ell_0-\ell_1$ and $S^2=-2(p+1)$.
\end{proposition}
\begin{proof}
Let $X$ be a fibre of $\pi_0$. We have  $X S=1$ and $(S-\ell_0)X=0$, hence $S-\ell_0=\pi_0^*(D)$ with
$D$ a divisor class on $\mathcal{F}_1$. This gives $(S-\ell_0)^2=0$.
The identity $\lambda_1^2=2\, \lambda_2$ implies $\ell_0^2=0$ and thus $S^2-2\ell_0S=0$.
Now we use the fact that $S$ is contracted under the map 
of $\mathcal{F}_0$ to $\mathcal{A}_3 \otimes {\FF}_p$.
This implies that $\lambda_1$ restricted to $S$ vanishes, hence $(\ell_0+\ell_1)S=0$.
We thus get $S^2=2\ell_0 S =-2\ell_1 S$ and on the other hand $S^2=\ell_0S+\pi_0^*(D)S=-\ell_1S+\pi_0^*(D)S$,
hence $\pi_0^*(D)=-\ell_1$ and $S=\ell_0-\ell_1$. The fact that $\ell_0\ell_1=p+1$ and $\ell_0^2=\ell_1^2=0$ implies 
$S^2=-2(p+1)$.
\end{proof}
\end{section}
\begin{section}{Loci for $g=3$ defined by conditions on the $a$-number}
We now discuss subloci of $S_3$ defined by the inequality $a\geq 2$. Here $a$ indicates the $a$-number
of an abelian variety.
Let $J$ with $\# J=N_3$ be the set of irreducible components of $S_3$ (where we count in the stacky way).
Each irreducible component of $S_3$ is the image under a morphism of a ${\PP}^1$-bundle
$\mathcal{F}_0 \to \mathcal{F}_1$ that blows down a section.
The curve $\mathcal{F}_1$ has $p^3+1$ points rational over ${\FF}_{p^2}$
and $\# \mathcal{F}_0({\FF}_{p^2})=(p^3+1)(p^2+1)$ and each
point of $\mathcal{F}_0({\FF}_{p^2})$ defines a superspecial
abelian variety. Let $\sqcup_{j \in J} \mathcal{F}_0^{j}$ be the
disjoint union of the smooth models of the irreducible components of $S_3$.
Under the natural morphism 
$$
m: \sqcup_{j \in J} \mathcal{F}_0^{j}\longrightarrow S_3 \subset 
\mathcal{A}_3 \otimes{\FF}_p\, .
$$
the $N_3(p^3+1)(p^2+1)$ superspecial points of $\sqcup_{j \in J} \mathcal{F}_0^{j}$
map to $N_3$ superspecial points of $S_3$.
Thus each superspecial point of $S_3$ is the image of
$(p^3+1)(p^2+1)$ points and this multiplicity can be explained as follows. On each surface $\mathcal{F}_0^j$
a section is contracted giving a factor $p^3+1$, while
the image of an ${\FF}_{p^2}$-rational fibre of $\mathcal{F}_0^j \to \mathcal{F}_1^j$ lies
on the image of 
$p^2+1$ surfaces $\mathcal{F}_0^j$. This can be checked by using Ekedahl-Oort strata and their classes 
as follows.  

Each ${\FF}_{p^2}$-rational point of $\mathcal{F}_1^j$ determines a fibre in the ${\PP}^1$-bundle 
$\mathcal{F}_0^j\to \mathcal{F}_1^j$ and the image under $m$ of such a fibre
provides a component of the Ekedahl-Oort locus $\mathcal{V}_{[3,2]}$.
This locus  $\mathcal{V}_{[3,2]}$ consists of a finite union of ${\PP}^1$s.
By \cite{E-vdG} we know the class of this locus:
$$
[\overline{\mathcal{V}}_{[3,2]}]=(p-1)^2(p^6-1) \, \lambda_2\lambda_3\, .
$$
Since the degree of the determinant $\lambda_1$ of the Hodge bundle
restricted to such a ${\PP}^1$ is $p-1$, we find that $\overline{\mathcal{V}}_{[3,2]}$
has
$$
m_{3,2}=\frac{\deg( [\overline{\mathcal{V}}_{[3,2]}]\, \lambda_1)}{p-1}=
(p-1)(p^6-1) \, v(3)
$$
irreducible components, each a copy of ${\PP}^1$. 
Here we count in the stacky sense.
Each such component contributes $p^2+1$ superspecial points and we see from
$$
m_{3,2} \, (p^2+1) = \deg(\overline{\mathcal{V}}_{[3,2,1]}) \, (p^3+1)
$$
that this fits with the fact that through a superspecial point there
pass $p^3+1$ components of $\overline{\mathcal{V}}_{[3,2]}$.
In fact, under the map $\mathcal{F}_0^j \to \mathcal{A}_3\otimes {\FF}_p$ a section is blown down
and this section intersects the $p^3+1$ fibres of $\mathcal{F}_0^j\to \mathcal{F}_1^j$ over
$\mathcal{F}_1^j({\FF}_{p^2})$.

We can also check that each such fibre lies on
$p^2+1$ irreducible components of $S_3$; hence we find for the number of
superspecial points
$$
N_3 \,  (p^3+1)(p^2+1) =
\deg({\overline{\mathcal{V}}_{[3,2,1]}}) (p^2+1)(p^3+1)\,
$$
in agreement with the fact that $\mathcal{V}_{[3,2,1]}$
is the superspecial locus and that
$N_g$ equals the degree of the superspecial locus for odd $g$.
\end{section}

\begin{section}{Moduli of flag type quotients for $g=4$} \label{g=4flagtypes}
In this section we construct a smooth model for each irreducible
component of the supersingular locus $S_4$. 
The model is obtained by taking 
the Zariski closure of the moduli of rigid 
flag type quotients for $g=4$ and by showing that this
moduli space is smooth. 

We consider the space $\mathcal{M}=\mathcal{M}_{\eta}$
of polarized flags of contravariant Dieudonn\'e
modules
$$
M_0 \subset M_1 \subset M_2 \subset M_3
$$
satisfying
\begin{enumerate}
\item{} $M_3=A_{1,1}^4$ provided with $\eta$, a fixed quasi-polarization $\langle \, , \, \rangle$ that induces an identification $M_3^t= F^3M_3$;
\item{} $(F,V)M_i \subset M_{i-1}$ and $\dim (M_i/M_{i-1})=i$;
\item{} $(F,V)^iM_i\subset M_i^t$.
\end{enumerate}
We say that it is {\sl rigid} if $M_i=M_0+F^{3-i}M_3$ for $i=0,\ldots,3$.

\begin{theorem}\label{thmZariskiclosure}
The Zariski closure $\mathcal{F}_0$ 
of the moduli space of rigid polarized Dieudonn\'e flags
of length $4$ with given quasi-polarization on $M_3$ inside
$\mathcal{M}$  is non-singular.
\end{theorem}
\begin{proof}
By \cite[6.1]{K-O2} we can choose generating elements $x_1,x_2,x_3,x_4$ of $M_3$
in the skeleton $\tilde{M}_3=\{ m \in M_3: (F-V)m=0\}$ of $M_3$ such that
the pairing defined by $\eta$ 
satisfies
$$
\langle x_i, F^4x_j\rangle = \delta_{i,5-j} \qquad \text{\rm and} \qquad
 \langle x_i,F^3x_j\rangle=0\, 
$$
for $1\leq i\leq j \leq 4$.
For a rigid polarized Dieudonn\'e flag $M$ the module $M_2$ is generated by
$FM_3$ and a vector 
$$
v_0=\sum_{i=1}^4 a_ix_i  \, \in M_3/FM_3 \label{eq:v_0}
$$
with the condition $\langle v_0,F^2v_0\rangle \in W$. Viewing the
coefficients $a_i$ as lying in $k$, this amounts to the equation
$$
f:= a_1a_4^{p^2}-a_1^{p^2}a_4+a_2a_3^{p^2}-a_2^{p^2}a_3=0\, .
$$
This defines a smooth surface $\mathcal{F}_2$ in ${\PP}^3$. This
surface was studied in detail by Katsura \cite{K2017}. 
Locally on this surface we may assume without loss of generality 
that $a_1\neq 0$ and that
$a_1=1$. Now $M_1$ is generated by $FM_2$ and a $2$-dimensional 
subspace $M_1/FM_2$ in $M_2/FM_2$. Since $a_1=1$ we can assume that
this $2$-dimensional subspace is generated by non-zero elements $v$ and
$w$ with
$$
v=a_5v_0+a_6Fx_2+a_7Fx_3+a_8Fx_4, \quad 
w=a_9Fx_2+a_{10}Fx_3+a_{11}Fx_4 \, . \eqno(6)
$$
We then have the conditions
$$
\langle v,Fv\rangle \in W,  \quad
\langle v,Fw\rangle \in W,  \quad
\langle Fv,w\rangle \in W\, . \eqno(7)
$$
Viewing the coefficients as elements of $k$ we find three equations 
all divisible by $a_5$. But $a_5=0$ yields a flag that is not rigid; indeed,
$$
M_1+FM_3=(F,V)M_2+Aw+FM_3=(F,V)(Av+FM_3)+Aw+FM_3 \subset FM_3
$$
but $M_2 \not\subset FM_3$, hence $M_2\neq M_1+FM_3$, contradicting rigidity.
Removing the factor $a_5$ from the equations (7) by considering
$\langle v,Fv\rangle/a_5, \langle v,Fw\rangle/a_5$ and
$\langle Fv,w\rangle/a_5^p$, we get the equations
$$
\begin{aligned}
g_1:=& a_1a_8^p-a_1^pa_5^{p-1}a_8+a_2a_7^p-a_2^pa_5^{p-1}a_7 +a_3^pa_5^{p-1}a_6-a_3a_6^p=0, \\
g_2:=& a_1a_{11}^p+a_2a_{10}^p-a_3a_9^p=0\, ,  \\
g_3:=& a_1^pa_{11}+a_2^pa_{10}-a_3^pa_{9}=0\, .\\
\end{aligned} \eqno(8)
$$

\begin{remark}
The reader may verify that if the point $(1:a_2:a_3:a_4) \in \mathcal{F}_2(k)$
is not rational over ${\FF}_{p^2}$ then we may choose as $w$
the element 
$$
(F-V)\, v_0 \, .
$$
Indeed, it satisfies $g_2=0$ and $g_3=0$ 
for any non-zero choice of $v$; namely
with $a_1=1$ we have $a_1(a_4^{p^2}-a_4)+a_2(a_3^{p^2}-a_3)-a_3(a_2^{p^2}-a_2)=0$
and similarly for $g_3$. 
\end{remark}
Now we first look at a point with $a_5\neq 0$. If both $a_9$ and $a_{10}$ 
vanish we have by $g_3=0$ that $w=0$. So we may assume that, say, 
$a_9\neq 0$
and then have $a_1=a_5=a_9=1$ and by changing $v$ to $v-a_6w$ we may assume $a_6=0$. 
The Jacobian matrix of the
equations $f,g_1,g_2,g_3$ with respect to the variables $a_j$ for
$j=2,3,4,7,8,10,11$ is
$$
\left( \begin{matrix}
a_3^{p^2} & -a_2^{p^2} & -a_1^{p^2} & 0 & 0 & 0 & 0\\
a_7^p & -a_6^p & 0 & -a_2^pa_5^{p-1} & -a_1^pa_5^{p-1} & 0 & 0 \\
a_{10}^p & -a_9^p & 0 & 0 & 0 & 0 & 0 \\
0 & 0 & 0 & 0 & 0 &a_2^p & a_1^p \\
\end{matrix}
\right)
$$
and this is of rank $4$. 

Next we look at the case where $a_5=0$. The vanishing of $a_9$ and $a_{10}$
implies by $g_3$ that $a_{11}=0$, so we may assume that $a_9\neq 0$ or 
$a_{10}\neq 0$. Again without loss of generality we may assume $a_9\neq 0$.
Changing $v$ by a multiple of $w$ we may assume $a_6=0$. If now $a_7=0$
then $g_1$ forces $a_8=0$, hence $v=0$. So we may assume that $a_7\neq 0$.
Then it suffices to treat the case of $a_1=a_7=a_9=1$ and $a_6=0$. Then the Jacobian matrix of the
equations $f,g_1,g_2,g_3$ with respect to the variables $a_j$ for
$j=2,3,4,5,8,10,11$ is
$$
\left( \begin{matrix}
a_3^{p^2} & - a_2^{p^2} & -a_1^{p^2} & 0  & 0 & 0 & 0 \\
a_7^p & 0 & 0 & (a_1^pa_8+a_2^pa_7)a_5^{p-2} & -a_1^pa_5^{p-1} & 0 & 0 \\
a_{10}^p & -a_9^p & 0 & 0 & 0 & 0 & 0 \\
0 & 0 & 0 & 0 & 0 & a_2^p & a_1^p \\
\end{matrix}
\right)
$$
which is of rank $4$ as required. This shows that
$\mathcal{F}_1$ is non-singular. Since $\mathcal{F}_0$ is a ${\PP}^1$-bundle
over $\mathcal{F}_1$ the result follows.
\end{proof}

By writing $\mathcal{F}_i$ for the Zariski closure in $\mathcal{M}$ 
of the moduli space
of rigid polarized Dieudonn\'e flags $M_i \subset \cdots \subset M_3$ we get
a sequence 
$$
\mathcal{F}_0 \langepijl{\pi_0} \mathcal{F}_1 \langepijl{\pi_1} \mathcal{F}_2 
\langepijl{\pi_2} \mathcal{F}_3={\rm point}
$$
with $\dim \mathcal{F}_i=4-i$ for $i=0,1,2$.

We now describe the fibres of the morphism $\pi_1 : \mathcal{F}_1 \to \mathcal{F}_2$.
We start by remarking that by using the symmetry of $\mathcal{F}_2$ 
there is no loss of generality if we look at the fibre
of a point $(a_1:a_2:a_3:a_4)$ of $\mathcal{F}_2$
with $a_1=1$. If one of $a_2,a_3,a_4$ lies in ${\FF}_{p^2}$
then the point lies on one of the lines of $\mathcal{F}_2$. Indeed, if $a_4
\in {\FF}_{p^2}$ then such a line is parametrically $(1:t:t:a_4)$, 
while if, say, $a_2 \in
{\FF}_{p^2}$ then such a line is $(t:a_2: 1: t)$. 

For describing the fibre 
over a point $(1:a_2:a_3:a_4)$ we consider the equations
$$
a_8^p+a_2a_7^p-a_3a_6^p -a_5^{p-1}(a_8+a_2^pa_7-a_3^pa_6)=0\, , 
\eqno(9)
$$
and
$$
a_{11}^p+a_2a_{10}^p-a_3a_9^p=0, \quad a_{11}+a_2^pa_{10}-a_3^pa_9=0\, . \eqno(10)
$$
By the two equations $g_2$, $g_3$ of (10) 
we eliminate $a_{11}$ and get 
$$
\frac{a_{10}^p}{a_9^p}= \frac{a_3^{p^2}-a_3}{a_2^{p^2}-a_2}\, .
\eqno(11)
$$
In the neighborhood of an ${\FF}_{p^2}$-valued point
of $\mathcal{F}_2$, say $(1:a_2:a_3:a_4)$, the expressions  $a_2-a_2^{p^2}$ and $a_3-a_3^{p^2}$
are local coordinates. This shows that the function field of
$\mathcal{F}_1$ can be generated over the function field of
$\mathcal{F}_2$ by adjoining the $p$th root of 
$(a_2-a_2^{p^2})/(a_3-a_3^{p^2})$ as determined by (11) and 
then adjoining a further element via an Artin-Schreier equation  (9).
Hence the degree of inseparability of $\mathcal{F}_1$ over
$\mathcal{F}_2$ is $p$.

Over an open neighborhood  $U$ of a
${\FF}_{p^2}$-rational point with 
local coordinates  $a_2-a_2^{p^2}$ and $a_3-a_3^{p^2}$,
the equation (11) describes
an inseparable cover of the blow-up of $U$ 
(in $U\times {\PP}^1$ with coordinates $(u:v)$ on ${\PP}^1$) given by
$$
u(a_2-a_2^{p^2})-v(a_3-a_3^{p^2})=0, \quad u/v=(a_{10}/a_9)^p\, .
$$
Thus we see that the morphism 
$\pi_1: \mathcal{F}_1 \to \mathcal{F}_2$
factors via an inseparable cover of the blow-up 
$\tilde{\mathcal{F}}_2$ of $\mathcal{F}_2$ in the ${\FF}_{p^2}$-rational points.

If we have a point not on a line we may assume $a_9=1$ and then that $a_6=0$.  The reduced
fibre is a curve in ${\PP}^2$ with coordinates $(a_5,a_7,a_8)$  given by 
$$
a_8^p+a_2a_7^p -a_5^{p-1}(a_8+a_2^pa_7)=0\, .
$$ 
This is a curve with one singularity of order $p-1$, a cusp located 
at $a_5=0$ and $a_8+a_2^{1/p}a_7=0$. 

Next we consider the case of a point on a line. Since the automorphism group
of $\mathcal{F}_2$ acts transitively on the set of lines defined over ${\FF}_{p^2}$ 
(by Witt's theorem, see \cite{K2017}) we may assume that the
line is given as $(1:t:0:0)$. The last two equations give
$(t^{p^2}-t)a_{10}^p=0$ and the first equation yields
$a_8^p+ta_7^p-a_5^{p-1}(a_8+t^pa_7)=0$, again a curve with a cusp.
So if the point is not a ${\FF}_{p^2}$-valued point of $\mathcal{F}_2$
we get $a_{10}^p=0$ and
as reduced fibre again a curve with a 
single singularity, a cusp.
If $t\in {\FF}_{p^2}$, then the first equation splits as the union of
$p$ lines passing through one point. 

We summarize.

\begin{proposition} Let $\tilde{\mathcal{F}}_2$ be the blow-up of $\mathcal{F}_2$
in all ${\FF}_{p^2}$-rational points. 
The morphism $\pi_1: \mathcal{F}_1 \to \mathcal{F}_2$ factors through
$ \mathcal{F}_1 \to \tilde{\mathcal{F}}_2 \to \mathcal{F}_2$. The
morphism $\pi_1': \mathcal{F}_1 \to \tilde{\mathcal{F}}_2$ has inseparability 
degree $p$. The reduced fibre over a 
non-${\FF}_{p^2}$-rational point is an irreducible curve with one singularity,
a cusp singularity of order $p-1$, while the fibre over a point on an 
exceptional curve is a union of $p$ lines
meeting in one point.
\end{proposition}
\end{section}
\begin{section}{Interpretation of the
morphism $\mathcal{F}_1 \to \mathcal{F}_2$}
The morphism $\pi_1: \mathcal{F}_1 \to \mathcal{F}_2$ is
inseparable and factors through the blown-up surface $\tilde{\mathcal{F}}_2$.
We give an interpretation of this factorization
by describing the blow-up
$\tilde{\mathcal{F}}_2$ in terms of Dieudonn\'e modules and by showing
that $\mathcal{F}_1$ is realized in a natural ${\PP}^2$-bundle
over $\tilde{\mathcal{F}}_2$.

 We begin with
a moduli interpretation of the fibers of 
$\tilde{\mathcal{F}}_2\to \mathcal{F}_2$.
\begin{proposition}\label{FiberOfBlowUp}
The fiber
of $\tilde{\mathcal{F}}_2 \to \mathcal{F}_2$ over a point 
$(M_2\subset M_3) \in {\mathcal{F}}_2(k)$ is given by a set
of lines in a $2$-dimensional vector space 
$$
\{L \subset V^{-1}M_2^t/FM_2 \mid \dim L = 1,\ 
\text{ $L$ contains $(F,V)M_2 \bmod FM_2$} \}.
$$
\end{proposition}
\begin{proof}
We begin by observing two facts:
\begin{enumerate}
\item[(i)] $(FM_2 \subset)\ (F,V)M_2 \subset V^{-1}M_2^t$
\item[(ii)] $V^{-1}M_2^t/FM_2$ is a $k$-vector space of dimension two.
\end{enumerate}
Indeed, (i) follows from $V(F,V) M_2 \subset (F,V)^2 M_2 \subset M_2^t$.
To prove (ii), consider  the dual of $FM_3 \subset M_2 \subset M_3$:
\[
M_3^t \subset M_2^t \subset V^{-1}M_3^t.
\]
By $V^{-1}M_3^t = F^{-1}M_3^t = F^2M_3 \subset FM_2$, we have $M_2^t \subset FM_2$.
This means that $V$ (and therefore $p$) kills  $V^{-1}M_2^t/FM_2$,
whence $V^{-1}M_2^t/FM_2$ is a $k$-vector space.
Looking at the inclusions $M_2^t \subset FM_2 \subset V^{-1}M_2^t$, we have
$$
\begin{aligned}
\dim V^{-1}M_2^t/FM_2 &= \dim V^{-1}M_2^t/M_2^t - \dim FM_2/M_2^t\\
&= 4- \dim FM_2/F^2M_3 - \dim V^{-1}M_3^t/M_2^t\\
&= 4-1-1 = 2 \\
\end{aligned}
$$
and this proves ii).
If $(M_2 \subset M_3)$ represents a point of $\mathcal{F}_2$ that is not
rational over ${\FF}_{p^2}$ then $FM_2\neq VM_2$ and $L$ is unique.
If $(M_2 \subset M_3)$ represents a ${\FF}_{p^2}$-rational point, then
$FM_2=VM_2$ and the fibre is a ${\PP}^1$. 
\end{proof}

\begin{remark}
We point out that the Dieudonn\'e module $V^{-1}M_2^t/FM_2$ is self-dual.
\end{remark}

We now describe the morphism $\pi_1' : \mathcal{F}_1 \to \tilde{\mathcal{F}}_2$.
On $\tilde{\mathcal{F}}_2$ we have by Proposition \ref{FiberOfBlowUp} 
the subspace $L \subset V^{-1}M_2^t/FM_2$. It determines a
$W$-module $\tilde{L}$ with
$$
(F,V) M_2 \subset \tilde{L} \subset V^{-1}M_2^t\, ,
$$
the inverse image of $L$ under the projection $V^{-1}M_2^t \to V^{-1}M_2^t/FM_2$.
It has the property that outside $\tilde{\pi}_2^{-1}(\mathcal{F}_2({\FF}_{p^2}))$ we have
$\tilde{L}=(F,V)M_2$, where we write $\tilde{\pi}_2$ for the blow-down morphism
$\tilde{\mathcal{F}}_2 \to \mathcal{F}_2$.
We can now consider over a point of $\tilde{\mathcal{F}}_2$ 
the $3$-dimensional vector space
$ M_2/\tilde{L}$. This should define
a rank $3$ vector bundle $B$, but as the equations show 
we can realize $B$ only after an inseparable base change.

\begin{lemma}
The threefold $\mathcal{F}_1$ is a divisor in a ${\PP}^2$-bundle ${\PP}(B)$
with $B$ the rank $3$ vector bundle defined by $M_2/\tilde{L}$
over a surface $\tilde{\mathcal{F}}_2^{\prime}$ obtained by an inseparable base change 
$\tilde{\mathcal{F}}_2^{\prime} \to\tilde{\mathcal{F}}_2$
of degree $p$.
\end{lemma}
\begin{proof}
Recall that in order to define $M_1 \subset M_2$, we chose a basis 
$$
v=a_5v_0+a_6Fx_2+a_7Fx_3+a_8Fx_4, \quad
w=a_9Fx_2+a_{10}Fx_3+a_{11}Fx_4 \, 
$$
as in (6) with $\langle v,Fv\rangle$, $\langle v,Fw\rangle$, $\langle Fv,w\rangle$
all in $W$. 
The equations ($g_2$) and ($g_3$) correspond to the inseparable 
base change $\tilde{\mathcal{F}}_2^{\prime} \to\tilde{\mathcal{F}}_2$
given on the locus with $a_1\neq 0$ by (11) 
$$
(a_9/a_{10})^p= (a_2-a_2^{p^2})/(a_3-a_3^{p^2})\, .
$$
Then on $\tilde{\mathcal{F}}_2^{\prime}$ we have the bundle ${\PP}(B)$.
If $a_5\neq 0$ the morphism
${\mathcal F}_1 \to \tilde{\mathcal F}_2$ is defined by sending
$(M_1 \subset M_2 \subset M_3)$ to the point defined by $L:=M_1 \cap V^{-1}M_2^t \bmod FM_2$.
Indeed,
by $(F,V)M_2 \subset M_1$, the subspace $L$ contains $(F,V)M_2 \bmod FM_2$,
and $L$ is the one-dimensional space generated by $w$ of (6),
since one can check $\langle Vw, M_2 \rangle\subset W$ and if $a_5 \ne0$, 
then $\langle Vv, M_2 \rangle \not\subset W$.

For $a_5\neq 0$ we find from $\langle v, Fv\rangle \in W$ an equation 
$$
a_1a_8^p+a_2a_7^p-a_3a_6^p -a_5^{p-1}(a_1^pa_8+a_2^pa_7-a_3^pa_6)=0\, .
$$
This defines a rational curve with a cusp in ${\PP}^2={\PP}(M_2/\tilde{L})$.
As $\mathcal{F}_1$ is defined as the closure of the space of rigid flags,
we obtain that this equation defines $\mathcal{F}_1$ in ${\PP}(B)$.
Observe that in order to analyze this we may assume as we did in the 
preceding section that
$a_1=1$ and $a_2\neq 0$ and then $a_6=0$ and the curve can be written 
in coordinates $(a_5:a_7:a_8)$ as
$$
a_1a_8^p+a_2a_7^p -a_5^{p-1}(a_1^pa_8+a_2^pa_7)=0\, .
$$
The cusp is determined by $a_5=0$ and $a_8+a_2^{1/p}a_7=0$. 
\end{proof}

In particular we see that after an inseparable base change the bundle $B$
admits a nowhere vanishing section.  
\end{section}
\begin{section}{The Hodge bundle on the supersingular locus}\label{Hodge-bundle}
The description of principally polarized supersingular
abelian varieties of dimension $4$
via a flag gives us for each irreducible component $S$ of $S_4$
a morphism $\mathcal{F}_0 \to S$ and a fibration of $\mathcal{F}_0$
$$
\mathcal{F}_0 \langepijl{\pi_0} \mathcal{F}_1 \langepijl{\pi_1} 
\mathcal{F}_2 \to \mathcal{F}_3\, ,
$$
where $\mathcal{F}_i$ for $i=0,\ldots,3$ is the closure of the moduli space of
rigid polarized flag type quotients $Y_3 \to \cdots \to Y_i$.
Note that $\mathcal{F}_3$ is a point. 

We have seen above that these spaces $\mathcal{F}_i$ are non-singular.
In the following we view these as moduli stacks. This corresponds to dividing
by the appropriate automorphism groups, here by ${\rm Aut}(Y_3,\eta)$.

\begin{lemma}\label{degree-p} For each irreducible component $S$ of $S_4$ in $\mathcal{A}_4\otimes {\FF}_p$ 
the natural morphism $\mathcal{F}_0\to S$ is a morphism of degree $p$.
\end{lemma}
\begin{proof}
Let $x$ be a geometric point of ${\mathcal F_2}$.
Let ${\mathcal F}_{1,x}$ be the fiber $\pi_1^{-1}(x)$.
We claim that $\pi_0^{-1}({\mathcal F}_{1,x}) \to \mathcal A_g$ is a $p$-to-$1$ map onto the image.
Indeed if $x$ is represented by $(a_1,a_2,a_3,a_4)\in k^4$ for an algebraically closed field $k$, then
the fiber $\pi_1^{-1}(x)$ is described in $a_5,\ldots,a_{11}$ by
$$
\begin{aligned}
g_1:=& a_1a_8^p-a_1^pa_5^{p-1}a_8+a_2a_7^p-a_2^pa_5^{p-1}a_7 +a_3^pa_5^{p-1}a_6-a_3a_6^p=0, \\
g_2:=& a_1a_{11}^p+a_2a_{10}^p-a_3a_9^p=0\, ,  \\
g_3:=& a_1^pa_{11}+a_2^pa_{10}-a_3^pa_{9}=0\, .\\
\end{aligned}
$$
But $g_2$ is the $p$-th power of
\[
g_2' := a_1^{1/p} a_{11} + a_2^{1/p} a_{10} - a_3^{1/p}a_9\, .
\]
The space defined by $g_1, g'_2, g_3$, say ${\mathcal F}'_{1,x}$, coincides on an open part of
$\mathcal{F}_1$ with the fiber 
of ${\mathcal V}_{11} \to {\mathcal V}_2$ studied in \cite[9.7]{L-O},
where ${\mathcal V}_2$ corresponds to our ${\mathcal F}_2$ and ${\mathcal V}_{11}$ is the non-garbage component considered in \cite[9.7]{L-O}.
Note that ${\mathcal F}'_{1,x}$ is a closed subscheme of ${\mathcal F}_{1,x}$.
Thanks to the proof by Li and Oort (cf.~\cite[7.11]{L-O}), the map $(\pi_0)^{-1}({\mathcal F}'_1) \to {\mathcal A}_g$ is one-to-one on its image as stacks; indeed,  the proof of Li and Oort was done by fiberwise arguments. The claim follows.
\end{proof}

The space $\mathcal{F}_i$ carries an abelian variety $\mathcal{Y}_i$.
Its cotangent bundle along the zero section may be described by Dieudonn\'e
theory. Using contravariant Dieudonn\'e theory with the Dieudonn\'e module $M_i$
of a fibre $Y_i$ of $\mathcal{Y}_i$, we have
$$
{\rm Lie}(Y_i)^{\vee}=VM_i/pM_i\, .
$$
The flag type quotient provides an inductive construction.
For $i=2,1,0$ we have the
exact sequence
$$
0 \to pM_{i+1}/pM_i \to VM_i/pM_i \to VM_{i+1}/pM_{i+1} \to VM_{i+1}/VM_i \to 0 \, .
$$
In the Grothendieck group of vector bundles on $\mathcal{F}_i$ we thus get the identity
$$
[{\rm Lie}(\mathcal{Y}_i)^{\vee}]=[\pi_i^*({\rm Lie}(\mathcal{Y}_{i+1})^{\vee})] - [Q_i] + [Q_i^{(p)}]\, ,
$$
where $Q_i$ is the locally free $\mathcal{O}_{\mathcal{F}_i}$-module
of rank $i+1$ corresponding to $VM_{i+1}/VM_i$.
Moreover, the exact sequence for $i=1$ and $i=2$
$$
0 \to VM_i/pM_{i+1} \to VM_{i+1}/pM_{i+1} \to VM_{i+1}/VM_i \to 0
$$
gives us an exact sequence of $\mathcal{O}_{\mathcal{F}_i}$-modules
$$
0 \to U_i \to \pi_i^*({\rm Lie}(\mathcal{Y}_{i+1})^{\vee}) \to Q_i \to 0
$$
with $U_i$ the locally free $\mathcal{O}_{\mathcal{F}_i}$-module
defined by  $VM_i/pM_{i+1}$. For $i=0$ we have
$$
0 \to VM_0/VM_1^t \to VM_1/VM_1^t \to VM_1/VM_0 \to 0
$$
and this gives a short exact sequence of $\mathcal{O}_{\mathcal{F}_0}$-modules
$$
0 \to U_0 \to \pi_0^*(K_1) \to Q_0 \to 0
$$
with $K_1$ the locally free sheaf corresponding to the Dieudonn\'e module
of $\ker(\mathcal{Y}_1 \langepijl{\eta_1} \mathcal{Y}_1^t)$.

In the following we will abuse the notation $Q_i$ also for the pullback of $Q_i$
to $\mathcal{F}_{i-1}$ in order to simplify notation.
Since ${\rm Lie}(\mathcal{Y}_3)^{\vee}$ is trivial of rank $4$ we get from the above the
class of the Hodge bundle ${\EE}={\rm Lie}(\mathcal{Y}_0)^{\vee}$ in the Grothendieck group
of vector bundle on $\mathcal{F}_0$.

\begin{proposition}\label{class of EE}
The class of Hodge bundle of $\mathcal{Y}_0$ in the Grothendieck group of
vector bundles on $\mathcal{F}_0$ is given by
$$
\begin{aligned}
 {} [{\EE}] 
& = 4 - [Q_2]+[Q_2^{(p)}]- [Q_1]+[Q_1^{(p)}]-[Q_0]+[Q_0^{(p)}]\\
& = 4 +[U_2]-[U_2^{(p)}]-[Q_1]+[Q_1^{(p)}]-[Q_0]+[Q_0^{(p)}]\, , \\
\end{aligned} 
$$
where $4$ stands for the class of trivial rank $4$ bundle and where
$U_2$ and $Q_0$ have rank $1$, while $Q_1$ has rank $2$.
\end{proposition}
Note that here we abuse the notation $Q_i$ for the pull back of $Q_i$ to $\mathcal{F}_0$.

We now set
$$
\ell_i= c_1(Q_i) \quad \text{\rm for $i=0,1,2$.}
$$
We may consider $\ell_i$
as a class living on $\mathcal{F}_i$, but we will denote the pull back $\pi_0^*(\ell_1)$, $\pi_1^*(\ell_2)$ and $\pi_0^*(\pi_1^*(\ell_2))$ also by $\ell_1$, $\ell_2$
in order to simplify notation.

Proposition \ref{class of EE} implies the following.

\begin{proposition}\label{Chern-Hodge-bundle}
The total Chern class $c({\EE})$
of the Hodge bundle on $\mathcal{F}_0$ is given by
$$
c({\EE})=
\frac{(1-\ell_2)(1+p\, \ell_1+p^2c_2(Q_1))(1+p\, \ell_0)}{(1-p\, \ell_2)(1+\ell_1+c_2(Q_1))(1+\ell_0)}\, .
$$
\end{proposition}
\begin{corollary}\label{Cor-c2Q1}
We have $c_2(Q_1)= (\ell_0^2+\ell_1^2-\ell_2^2)/2$. Moreover,
the class $\ell_0^2$ is a pullback from $\mathcal{F}_1$.
\end{corollary}\label{c2Q1}
\begin{proof}
We deduce
$\lambda_1=(p-1)(\ell_0+\ell_1+\ell_2)$
and
$$
2\,\lambda_2 -\lambda_1^2 =(p^2-1)(2\, c_2(Q_1)-\ell_0^2-\ell_1^2+\ell_2^2)
$$
and since $\lambda_1^2=2\, \lambda_2$ on $\mathcal{A}_4$ the formula for
$c({\EE})$ follows.
Since $Q_1$ lives on $\mathcal{F}_1$ it implies
that the class $\ell_0^2$ is a pullback from $\mathcal{F}_1$.
\end{proof}
\end{section}
\begin{section}{Loci with $a$-number $\geq 2$ for $g=4$}\label{section-a-greater-2}
The abelian variety corresponding to the generic point of an irreducible component $S$
of $S_4$ has $a$-number equal to $1$. An irreducible component of the closed stratum of 
$S$ where $a\geq 2$ is of one of two types, as shown in \cite[Section 9.9]{L-O}. See also \cite{Harashita2}. 
A component of the preimage of the first type for the natural morphism
$\mathcal{F}_0\to S$ maps under $\mathcal{F}_0 \to \mathcal{F}_2$ to a line
on $\mathcal{F}_2$, while such a component of the second type 
maps either dominantly to $\mathcal{F}_2$, or maps to a line of $\mathcal{F}_2$, or
maps to a point of $\mathcal{F}_2({\FF}_{p^2})$.
\bigskip

\subsection{Loci of the first type.}
The first type parametrizes flag types
$M_3 \supset M_2 \supset M_1 \supset M_0$ such that there exists a totally isotropic subspace
$I$ of $M_3/FM_3$ such that $M_1 \subset N$ with $N\subset M_3$ the submodule 
generated by $I$ and $FM_3$.
Since the automorphism group of $M_3$ acts transitively on totally isotropic
subspaces defined over ${\FF}_{p^2}$, we may assume that $I=\langle x_1,x_2\rangle$. 
In terms of abelian varieties, such a flag type can be obtained from a flag type
$$
E^4=Y_3 \langepijl{\rho_3} Y_2 \langepijl{\rho_2} Y_1 
\langepijl{\rho_1} Y_0 \eqno(12)
$$
with quasi-polarization $\eta_3: Y_3 \to Y_3^t$ with 
$\ker \eta_3=E^4[F^3]$ if the composition
$\rho_2 \rho_3: E^4 \to Y_1$ factors through
$$
1_{E^2}\times F_{E^2}: E^4 \longrightarrow E^2 \times E^2/E^2[F]\, .
$$
By identifying $E^2 \times E^2/E^2[F]$ with $E^4$ and thus factoring 
$\rho_2\rho_3$, 
we put $Z_2=E^2 \times E^2/E^2[F] \cong E^4$ and $Z_1=Y_1$
and then associate to it the flag
$$
E^4=Z_2 \langepijl{\zeta_2} Z_1 \langepijl{\zeta_1} Z_0 \, , \eqno(13)
$$
where $\deg(\zeta_2)=p^3$ and $\deg(\zeta_1)=p$ and 
$\theta_2: Z_2 \to Z_2^t$ is a quasi-polarization
with kernel equal to $E^4[p]$. 

This can be described by Dieudonn\'e modules: 
consider the Dieudonn\'e module
 $N_2=\langle x_1,x_2,Fx_3,Fx_4\rangle$ with 
$x_1,x_2,x_3,x_4$ the skeleton of $M_3$.
It satisfies $N_2^t=F^2N_2$. We choose a submodule $N_1$ generated by 
$u=ax_1+bx_2+cFx_3+dFx_4$ and $FN_2$ with $\langle u,Fu\rangle=0$.
By viewing $u$ as an element of $N_2/FN_2$ and the coefficients $a,b,c,d$ in $k$ we
obtain an equation
$$
ad^p-a^pd+bc^p-b^pc=0 \, . \eqno(14)
$$
Then $\dim N_2/N_1=3$ and $\dim N_1/N_1^t=2$. We then can choose a Dieudonn\'e
submodule $N_0$ with $N_1^t \subset N_0 \subset N_1$ with $\dim N_1/N_0=1$.
The filtration $N_0\subset N_1 \subset N_2$ corresponds to (13).
The moduli of $N_2 \supset N_1$ 
defines a surface $\mathcal{G}_1$ in projective space ${\PP}^3$ 
given by (14) 
and choosing $N_0$ defines a ${\PP}^1$-bundle $\mathcal{G}_0 \to \mathcal{G}_1$. 

\bigskip

Let $\mathcal{S}_1$ be the subscheme of $\mathcal{G}_1$ where $a=b=0$ and let
$\mathcal{S}_0$ be the inverse image of $\mathcal{S}_1$ under $\mathcal{G}_0\to \mathcal{G}_1$.
We now discuss how to map $\mathcal{G}_1\backslash \mathcal{S}_1$ to $\mathcal{F}_1$. Given $u$
we choose $v_0$ as a multiple of $u$. This determines a submodule $M_2$
of $M_3$,  generated by $FM_3$ and $v_0$, 
that contains $N_1$ and we set $M_1=N_1$. Note that $M_2$ is generated
also by $ax_1+bx_2$ and $FM_3$. 
Then we can choose two generators $v,w$ for $M_1$ modulo $FM_2$
and assuming that $a\neq 0$ we may choose
$$
v=a_5v_0+cFx_3+dFx_4=ax_1+bx_2+cFx_3+dFx_4 , \quad w=Fx_2\, .
$$
In terms of the coordinates in Section \ref{g=4flagtypes} we have
$$
a=a_1a_5,\quad b=a_2a_5, \quad c=a_7, \quad d=a_8\, .
$$
The fibre of $\mathcal{G}_1 \backslash \mathcal{S}_1\to \mathcal{F}_1$ 
over a point $v_0=(1:t:0:0)$ of $\mathcal{F}_2$ consists of all
$(a:b:c:d)$ with $d^p-a^{p-1}d+tc^p-t^pa^{p-1}c=0$; 
it is defined by a 
Lefschetz pencil on $\mathcal{G}_1$ defined by $b=ta$. We refer
to the paper \cite{K2017} for such a Lefschetz fibering.
The general fibre is a rational
curve with one singularity given by $a=0$.

Recall that the automorphism group 
of $\mathcal{F}_2$ 
acts transitively on the set of lines of $\mathcal{F}_2$ defined
over ${\FF}_{p^2}$.
For each line $L$ defined over ${\FF}_{p^2}$ 
on the surface $\mathcal{F}_2$ we find a surface
isomorphic to $\mathcal{G}_1$ 
that is contained in the inverse image $\pi_1^{-1}(L)$. 

The fibration $\mathcal{G}_0 \to \mathcal{G}_1$ has a 
natural section $S$ by taking $Z_0=Z_2/Z_2[F]$. 
Note that then $N_0=FN_2 \subset N_1$. 
This implies that $Z_0$, determined by
$N_0$, is constant for all choices of $N_1$. It also implies that this 
section is blown down under the natural morphism $\mathcal{G}_0 \to S_4 \subset 
\mathcal{A}_4$ that associates to a flag type quotient (13) the isomorphism
class of $Z_0$. We summarize: 

Let $M_3=A_{1,1}^4$ with quasi-polarization such that $M_3^t=F^3M_3$. 
 
\begin{proposition} For each totally isotropic subspace of $M_3/FM_3$ 
there is a threefold $\mathcal{G}_0$ that is a ${\PP}^1$-bundle 
$\mathcal{G}_0 \to \mathcal{G}_1$ over a surface given by (14) with a section 
and a morphism $\mathcal{G}_0\backslash \mathcal{S}_0 \to \mathcal{F}_0$ whose image is a locus of supersingular
abelian $4$-folds with $a\geq 2$. Under $\mathcal{G}_0 \backslash \mathcal{S}_0 \to 
\mathcal{F}_2$ it maps to a line on $\mathcal{F}_2$.
Under the morphism $\mathcal{G}_0 \to \mathcal{A}_4\otimes {\FF}_p$
the section of $\mathcal{G}_0 \to \mathcal{G}_1$ is blown down.
\end{proposition}
\subsection{Loci of the second type}\label{g=4secondtype}
An irreducible component of the preimage under
$\mathcal{F}_0 \to S$ 
of the locus of $a$-number $\ge 2$ of the second type inside
an irreducible component $S$ of the supersingular locus $S_4$
is realized as follows. It is the locus of $M_{\bullet}$ 
with $M_0 \subset N \subset M_3$ for a fixed superspecial 
quasi-polarized Dieudonn\'e module $N$ with $N^t=FN$.
Such $N$ come in three sorts: the first sort with $N=FM_3$,
the second with $N=\langle x_1,Fx_2,Fx_3,px_4\rangle$ and the
third sort where $N=\langle x_1,x_2,px_3,px_4\rangle$ after we change
generators as in the proof of Theorem \ref{thmZariskiclosure}.

In this subsection we treat the case where $N$ is of the first sort,
while the case of the second sort in treated in the next section
and the last case is left to the reader. The first case is characterized by the condition $M_1 \subset FM_3$.
This condition is determined on $\mathcal{F}_1$. 

Let $T$ be an irreducible component in $\mathcal{F}_0$ of the preimage of 
an irreducible component of the locus with $a \geq 2$ of the second type and first sort.

The condition $M_1 \subset FM_3$
can be paraphrased by saying 
that the natural homomorphism
$$
M_1/(F,V)M_2 \to M_2/FM_3 \eqno(15)
$$
induced by $M_1 \hookrightarrow M_2$ is zero.
Let $\mathcal L$ be the sheaf corresponding to the module
$M_1/(F,V)M_2$ and $U_2$ the one corresponding to $ VM_2/pM_3$. 
The invertible sheaf $U_2$ lives on $\mathcal{F}_2$, 
and $\mathcal{L}$ lives on $\mathcal{F}_1$ and  is invertible only
outside $\pi_1^{-1}(\mathcal{F}_2({\FF}_{p^2}))$.
Thus we work on the open set $\mathcal{F}_1^0$ that is the complement of
$\pi_1^{-1}(\mathcal{F}_2({\FF}_{p^2}))$.
The homomorphism (15) defines a homomorphism of sheaves
$$
\psi: \mathcal{L} \to \pi_1^*(U_2^{(p)})\, .
$$
The locus $\mathcal{H}_1$ can now be defined as the Zariski closure in
$\mathcal{F}_1$ of the zero locus $D(\psi)$ in $\mathcal{F}_1^0$
of the map $\psi$.

\begin{lemma}\label{Dpsi-class}
The cycle class of the Zariski closure in $\mathcal{F}_1$ 
of the zero locus $D(\psi)$ of $\psi$
equals
$$
[\overline{D(\psi)}]=p\, \ell_1-(p^2+1)\ell_2+e\, ,
$$ 
where $e$ is a class with support
in the fibres of $\pi_1$ over $\mathcal{F}_2({\FF}_{p^2})$.
\end{lemma}
\begin{proof} We work on the open set $\mathcal{F}_1^0$ that is the complement of 
$\pi_1^{-1}(\mathcal{F}_2({\FF}_{p^2}))$.
Consider the exact sequence
$$
0 \to VM_2/FM_2\cap VM_2 \to M_1/FM_2 \to M_1/(F,V)M_2 \to 0\, .
$$
If $M_2$ is generated by $FM_3$ and $v_0=a_1x_1+a_2x_2+a_3x_3+ a_4x_4$ with $(a_1:a_2:a_3:a_4)$
determining a point 
not in $\mathcal{F}_2({\FF}_{p^2})$ then $FM_2 \cap VM_2=pM_3$ and
$$
VM_2/FM_2\cap VM_2 =VM_2/pM_3. 
$$
This translates into a short exact sequence of $\mathcal{O}_{\mathcal{F}_1^0}$-modules
$$
0 \to \pi_1^*(U_2) \to U_1^{(p)} \to \mathcal{L} \to 0 
$$
with $U_1$ the locally free $\mathcal{O}_{\mathcal{F}_{1}}$-module
determined by $VM_1/pM_2$.
We view $\psi$ as a section of
$\pi_1^*(U_2^{(p)})\otimes \mathcal{L}^{-1}$ on $\mathcal{F}_1^0$ with class
$(p+1)[U_2]-p[U_1]$. 
From Section \ref{Hodge-bundle} we use the identities
$[U_2]=[4]-[Q_2]$,
$[U_1]=[{\rm Lie}(Y_2)^{\vee}]-[Q_1]=[4]-[Q_2]+[Q_2^{(p)}]-[Q_1]$,
hence $c_1((p+1)[U_2]-p[U_1])=p\ell_1-(p^2+1)\ell_2$.
When taking the closure of the
degeneracy locus of $\psi$ we have to take into account a class with
support in the fibres over $\mathcal{F}_2({\FF}_{p^2})$ and the result follows.
\end{proof}

\begin{corollary} A locus $T$ of the second type and first sort as above 
maps dominantly to $\mathcal{F}_2$.
\end{corollary}
\begin{proof}
We claim that such a $T$ is not contained in the fibres over $\mathcal{F}_2({\FF}_{p^2})$.
Indeed, otherwise $T$ lives inside a ${\PP}^1$-bundle over the one-dimensional
locus in $\mathcal{F}_1$ defined by $(a_1:a_2:a_3:a_4)\in {\PP}^3({\FF}_{p^2})$
and $a_5=0$. This contradicts that the stratum of $S$ where $a\geq 2$ has dimension $3$,
see \cite[9.9]{L-O}. The claim implies 
that $T$ is contained in
$\overline{D(\psi)}$ and the class is given by Lemma \ref{Dpsi-class}.
The intersection number of $\overline{D(\psi)}$ with a generic fibre of
$\pi_1$ equals the degree of $\ell_1$ on such a fibre, that is, $1$.
That means that it intersects the generic fibre
of $\pi_1$ and is irreducible.
\end{proof}
The resulting abelian variety of a flag type with $M_1 \subset FM_3$ 
is defined by the filtration of Dieudonn\'e
modules $FM_3 \supset M_1 \supset M_0$. By forgetting $M_1$ between
$FM_3$ and $M_0$,  we have polarized flag types
$$
\begin{xy}
\xymatrix{
E^4=Z_1 \ar[r]^{\rho_1} \ar[d]^{\eta_1}  & Z_0 \ar[d]^{\cong} \\
E^4=Z_1^t & Z_0^t \ar[l] \\
}
\end{xy}
$$
with $\ker(\eta_1)=E^4[F]$ and $\ker(\rho_1)\cong \alpha_p^2$. The choice of
$\ker(\rho_1)$ in $E^4[F]$ defines a point in the Grassmann variety $G={\rm Gr}(2,4)$.
Note that $G$ can be identified with a quadric in ${\PP}^5$ in terms of Pl\"ucker coordinates.
If we choose a basis $x_1,x_2,x_3,x_4$ of the Dieudonn\'e module of $E^4$
with $(F-V)x_i=0$ and quasi-polarization with
$$
\langle x_i, p\, x_{5-j}\rangle = \delta_{ij}, \quad \langle x_i,Fx_j\rangle=0 \, ,
$$
then $M_0$ can be generated by two vectors $a=\sum_{i=1}^4 a_ix_i$,
$b=\sum_{i=1}^4 b_i x_i$ and the condition $\langle a,b\rangle \in W$ says
$$
a_1b_4-a_4b_1+a_2b_3-a_3b_2\equiv \, 0 \, (\bmod \, p)
$$
and this defines a hyperplane section $\mathcal{Q}=H \cap G$ of the Grassmann variety.
Indeed, with the Pl\"ucker coordinates $\lambda_{ij}=a_ib_j-a_jb_i$ the variety
$G$ is given by $\lambda_{12}\lambda_{34}-\lambda_{13}\lambda_{24}+\lambda_{14}\lambda_{23}=0$
and $H$ by $\lambda_{14}+\lambda_{23}=0$.

Recall that we interpret moduli in the stacky way meaning that we divide by
the automorphism groups of objects.  We summarize.

\begin{lemma}
The image in $S_4$ 
of an irreducible component of the locus of second type with $a\geq 2$ 
can be identified with (the quotient of) a hyperplane section $\mathcal{Q}$ 
of the Grassmann variety ${\rm Gr}(2,4)$.
\end{lemma}

We will analyze the case of loci of the second type contained in the fibres
over ${\FF}_{p^2}$-rational points on $\mathcal{F}_2$ in the next section.
\end{section}
\begin{section}{The fibres over $\mathcal{F}_2({\FF}_{p^2})$}
Here we study the fibre under $\mathcal{F}_0 \to \mathcal{F}_2$ 
of a rational point $\xi \in \mathcal{F}_2({\FF}_{p^2})$. 
Since $\mathcal{F}_0 \to \mathcal{F}_1$ is a ${\PP}^1$-bundle it suffices to study
the fibre under $\mathcal{F}_1 \to \mathcal{F}_2$.

The automorphism group ${\rm Aut}(\mathcal{F}_2)$ can be identified with the
quotient by its center of the general unitary group ${\rm GU}_4(p^2)$ of 
$4$-dimensional space 
over ${\FF}_{p^2}$ that fixes the Hermitian form
$$
\xi_1 \bar{\xi}_4-\xi_4\bar{\xi}_1+\xi_2\bar{\xi}_3-\xi_3\bar{\xi}_2
$$
where $\bar{\xi}=\xi^{p^2}$. By a theorem of Witt this group acts transitively
on isotropic subspaces of dimension $1$ and $2$. 
This implies that it acts transitively on the set of
lines of $\mathcal{F}_2$ and on the set of ${\FF}_{p^2}$-rational points, see \cite[Appendix]{K2017}.
We thus may restrict to analyzing the fibre over the point $(1:0:0:0)$ of $\mathcal{F}_2$.
This corresponds to the case with $M_2\subset M_3$ generated by $v_0=x_1$ and
$FM_3$. The fibre $\pi_1^{-1}(\xi)$ corresponds to the choices of $M_1$.
It can be given by a choice of basis
$$
v=a_5v_0+a_6Fx_2+a_7Fx_3+a_8Fx_4, \quad w= a_9Fx_2+a_{10}Fx_3+a_{11}Fx_4,
$$
satisfying the equations $g_1$, $g_2$ and $g_3$ of Section \ref{g=4flagtypes}
$$
a_8^p-a_5^{p-1}a_8=0, \quad a_{11}^p=0, \quad a_{11}=0 \, .
$$
We distinguish whether $a_5\neq 0$ or $a_5=0$. 

{\sl Case i).} $a_5\neq 0$. We may assume $a_5=1$ and find $a_8\in {\FF}_{p}$.
We can change basis of $M_3$ by $(x_1,x_2,x_3,x_4) \mapsto
(x_1+a_8Fx_4,x_2,x_3,x_4)$ and then may assume that $a_8=0$. Then
$M_1$ is generated inside $M_2$ by $FM_2=\langle Fx_1,F^2x_2,F^2x_3,F^2x_4\rangle$
and $v=x_1+a_6Fx_2+a_7Fx_3$ and $w=a_9Fx_2+a_{10}Fx_3$. We now
construct a flag of Dieudonn\'e modules
$$
F^2M_3^{\prime} \subset M_1^t \subset M_1 \subset FM_3^{\prime}
$$
with $M_3^{\prime}=\langle F^{-1}x_1,x_2,x_3,Fx_4\rangle$
and show that we can extend it to a flag type
$$
M_1 \subset FM_3^{\prime} \subset M_2^{\prime} \subset M_3^{\prime}
\eqno(16)
$$
so that we can associate to it a point of a locus of the second type as
treated in the Section \ref{g=4secondtype} with respect to a changed basis
$\langle F^{-1}x_1,x_2,x_3,Fx_4\rangle$ of $M_3$. To prove our claim
we have to construct $M_2^{\prime}$ with $(F,V)M_2^{\prime} \subset M_1$.
We take $v_0^{\prime}= \alpha_1 F^{-1}x_1+\alpha_2 x_2 +\alpha_3 x_3 + Fx_4$
and impose the following conditions
\begin{enumerate}
\item{} $\langle v_0',Fv_0'\rangle \in W$, that is,
$\alpha_1-\alpha_1^{p^2}+\alpha_2\alpha_3^{p^2}-\alpha_2^{p^2}\alpha_3=0$,
\item{} $Fv_0' \in M_1$, equivalently, there exists $\beta$ with
$Fv_0'=\alpha_1^pv +\beta w+ F^2x_4$, that is, $\alpha_2^p=\alpha_1^pa_6+\beta a_9$
and $\alpha_3^p=\alpha_1^pa_7+\beta a_{10}$,
\item{} $Vv_0'\in M_1$, equivalently, there exists $\gamma$ with
$Vv_0'=\alpha_1^{1/p}v+\gamma w +F^2x_4$, that is, 
$\alpha_2^{1/p}=\alpha_1^{1/p}a_6+\gamma a_9$ and $\alpha_3^{1/p}=\alpha_1^{1/p}a_7
+\gamma a_{10}$.
\end{enumerate}
For generic $a_i$ (that is, $a_7a_9-a_6a_{10}$ and $a_9a_{10}$ not in ${\FF}_{p^2}$)
we find a solution. 
We then set
$$
M_2'=Av_0'+FM_3'
$$
and then by (2) we have $\langle v,w,FM_2'\rangle=\langle v,w,FM_2\rangle=M_1$.
Thus we have a filtration (16) and 
it gives a point of a locus $\mathcal{H}_1$ with respect 
to the module $M_3'$. 

{\sl Case ii).} Here $a_5=0$. Then by $g_2$ we have $a_{11}=0$ and find that $M_1$
is generated by $v=a_6Fx_2+a_7Fx_3$, $w=a_9Fx_2+a_{10}Fx_3$ and $F^2M_3$, hence
$M_1=\langle Fx_1,Fx_2,Fx_3,F^2x_4\rangle$. So $M_1$ is fixed and this case thus yields one point.
Moreover $M_1^t=\langle Fx_1,F^2x_2,F^2x_3,F^2x_4\rangle$.  

We thus see that the supersingular abelian variety 
corresponding to a generic point of
an irreducible component $\mathcal{E}$ of the fibre over a rational point
$\xi \in \mathcal{F}_2({\FF}_{p^2})$ 
can be viewed as the supersingular abelian variety defined by a generic point
of a locus with $a\geq 2$ of the second kind with $M_1 \subset FM_3^{\prime}$.

This means that there is an irreducible component $S'$ of $S_4$
with model $\mathcal{F}_0^{\prime}$ 
and a locus $\mathcal{H}_0^{\prime}$ 
mapping dominantly to $\mathcal{F}_2^{\prime}$ such that image of $\mathcal{E}$
and $\mathcal{H}_0^{\prime}$ coincide in $S_4 \subset \mathcal{A}_4$. 

We summarize.

\begin{proposition} Let $S$ be a component of $S_4$ and $\mathcal{F}_0$ be the
model constructed in Section \ref{g=4flagtypes}.
The fibre in $\mathcal{F}_0$ over a rational point $\xi \in \mathcal{F}_2({\FF}_{p^2})$
consists of $p$ irreducible components. The image of each of these in $S_4$ is a hyperplane
section of the Grassmann variety ${\rm Gr}(2,4)$ and can be seen as the image of 
a locus of $a\geq 2$ of the second type in another component $S'$ of $S_4$.
\end{proposition}
\end{section}
\begin{section}{Superspecial points of $S_4$}
The number of points of $S_4$ representing isomorphism classes of superspecial
abelian varieties counted in the stacky sense was given 
in formula (2) in Section \ref{Space-Ag}  and equals
$$
\Sigma_4= (p-1)(p^2+1)(p^3-1)(p^4+1) \, v(4)\, .
$$
Each superspecial principally polarized abelian variety of dimension $4$ defines an
${\FF}_{p^2}$-rational point of $S_4\subset \mathcal{A}_4$. By Proposition
\ref{Ng} we have $N_4=(p^2-1)(p^6-1)\, v(4)$
irreducible components (again counted in the stacky sense) of $S_4$. Each irreducible component
is the image of $\mathcal{F}_0$ under a degree $p$ morphism in the stacky sense to its image in $S_4$ 
that induces a bijection between geometric points of the stacks on the open parts of $a$-number one.
\begin{lemma}
We have $ \# \mathcal{F}_0({\FF}_{p^2})=(p^2+1)^3(p^3+1)(p^4+1)$.
\end{lemma}
\begin{proof}
We have $\# \mathcal{F}_2({\FF}_{p^2})=(p^2+1)(p^4+1)$, see for example \cite{K2017}, hence
$\# \tilde{\mathcal{F}}_2({\FF}_{p^2})=(p^2+1)^2(p^4+1)$ and these points are the
${\FF}_{p^2}$-rational points on the exceptional curves of $\tilde{\mathcal{F}}_2$.
The fibre in $\mathcal{F}_1$ over a ${\FF}_{p^2}$-rational point of $\tilde{\mathcal{F}}_2$
consists of a union of $p$ lines through one point. So we find
$\# \mathcal{F}_1({\FF}_{p^2})=(p^2+1)^2(p^4+1)(p^3+1)$. Since $\mathcal{F}_0$ is a ${\PP}^1$-bundle
over $\mathcal{F}_1$ the formula follows.
\end{proof}
Let $J$ be the set of irreducible components of $S_4$ and for $j\in J$ we let $\mathcal{F}_0^j$
be the corresponding smooth model. The disjoint union of these smooth models has
$$
\# (\bigsqcup_{j\in J} \mathcal{F}_0^j)({\FF}_{p^2})= N_4 \, (p^2+1)^3(p^3+1)(p^4+1) 
$$
${\FF}_{p^2}$-rational points mapping to $\Sigma_4$ superspecial points of $S_4$.
The variety $\mathcal{F}_0$ contains $(p^2+1)(p^4+1)$ loci $\mathcal{G}_0^n$ of the
first kind, each isomorphic to $\mathcal{G}_0$. We have
$\# \mathcal{G}_1({\FF}_{p^2})=(p^2+1)(p^3+1)$ (see \cite{K2017})
and $\# \mathcal{G}_0({\FF}_{p^2})=(p^2+1)^2(p^3+1)$ since $\mathcal{G}_0$
is a ${\PP}^1$-bundle over $\mathcal{G}_1$.
On $\mathcal{F}_1$ these loci $\mathcal{G}_1$ of the first kind are disjoint and we
see
$$
\# \mathcal{F}_0({\FF}_{p^2})=(p^2+1)(p^4+1)\,  \# \mathcal{G}_0({\FF}_{p^2})\, .
$$
On each component $\mathcal{G}_0^n$ a section of $\mathcal{G}_0 \to \mathcal{G}_1$
is blown down. This section has $(p^2+1)(p^3+1)$ points rational over ${\FF}_{p^2}$.
\begin{lemma}
Each superspecial point of $S_4$ lies on $(p+1)(p^3+1)$ irreducible components of $S_4$.
\end{lemma}
\begin{proof}
The number of totally isotropic subspaces of dimension $2$ in a $4$-dimensional unitary
space over ${\FF}_{p^2}$ with conjugation given by Frobenius is  equal to $(p+1)(p^3+1)$. A choice of an irreducible component corresponds
exactly to the choice of a totally isotropic subspace.
\end{proof}
We thus see that under the natural map
$$
\bigsqcup_{j\in J} \mathcal{F}_0^j \longrightarrow S_4
$$
the inverse image of each of the $\Sigma_4$ superspecial points of $S_4$
has 
$$
(p+1)(p^3+1) \times (p^2+1)(p^3+1) \times (p^2+1)
$$ points, where the second factor corresponds to blowing down the section
of $\mathcal{G}_0 \to \mathcal{G}_1$, and the third one comes from the fact
that each exceptional curve on  $\tilde{\mathcal F}_2$ intersects $p^2+1$
proper images of the lines defined over ${\FF}_{p^2}$, in agreement with the formula
$$
N_4 \, (p^2+1)^3(p^3+1)(p^4+1)= \Sigma_4 \, (p+1)(p^2+1)^2(p^3+1)^2\, .
$$

\end{section}

\begin{section}{The cycle class of $S_4$ and intersection numbers}
In this section we express the cycle class of the supersingular
locus $S_4$ for dimension $g=4$ in terms of intersection numbers. 

We know that the cycle class of $S_4$ lies in the tautological ring
and is a multiple of $\lambda_4\lambda_2$. This multiple can be 
determined by intersection numbers. We identify the degree of
a top-dimensional Chern class with an intersection number.

\begin{proposition}\label{lambda4lambda2}
We have
$[S_4]= a \, \lambda_4\lambda_2$ with
$$
a= \frac{\lambda_3\lambda_1 \, [S_4]}{v(4)} = 
\frac{\lambda_1^4 [S_4]}{8\, v(4)}
$$ with
$v(4)$ the proportionality constant defined in Section \ref{Space-Ag}.
\end{proposition}
\begin{proof}
We have $\lambda_3\lambda_1[S_4]=
a \, \lambda_4\lambda_3\lambda_2\lambda_1=a \, v(4)$. 
In the tautological ring $R_4$ we have $\lambda_3\lambda_1=\lambda_1^4/8$.
\end{proof}
We shall calculate the intersection number $[S]\cdot \lambda_3\lambda_1$
for each irreducible component $S$ of $S_4$. We will do this
by pulling back the Hodge bundle of $\mathcal{A}_4$ to 
$\mathcal{F}_0$ and calculating the degrees of the top Chern
classes of the Hodge bundle on $\mathcal{F}_0$. 

\end{section}
\begin{section}{Determination of intersection numbers}
Our goal is to calculate the intersection number 
$\lambda_1\lambda_3[S]$
for each irreducible component  
$S$ of the supersingular locus.
For this we calculate $\deg (\lambda_3\lambda_1)$ on the $4$-dimensional variety $\mathcal{F}_0$. 

By Proposition \ref{Chern-Hodge-bundle}, 
which describes the total Chern class of the Hodge bundle, and by Corollary \ref{Cor-c2Q1}
this intersection number can be expressed
in the intersection numbers given by the monomials of degree $4$ in $\ell_0,\ell_1,\ell_2$
evaluated at the fundamental class of $\mathcal{F}_0$. 
Note that we write $\ell_1$ and $\ell_2$ for their
pullbacks to $\mathcal{F}_0$ and sometimes identify such a monomial
$\ell_0^{a}\ell_1^b\ell_2^c$
with $\deg(\ell_0^{a}\ell_1^b\ell_2^c)$. 

\begin{lemma} \label{vanishing-monomials}
The following intersection numbers vanish on 
$\mathcal{F}_0$:
$$
\ell_0^4, \ell_0^2\ell_1^2, \ell_0^2\ell_1\ell_2, \ell_0^2\ell_2^2, \ell_0\ell_2^3,
\ell_1^4, \ell_1^3\ell_2, \ell_1^2\ell_2^2,
\ell_1\ell_2^3, \ell_2^4 \, .
$$ 
\end{lemma}
\begin{proof}
Since $\dim \mathcal{F}_1=3$ and $\ell_0^2$ is a pullback from $\mathcal{F}_1$ by Corollary \ref{c2Q1},
and $\ell_1$ and $\ell_2$ are also pullbacks from $\mathcal{F}_1$ we find
that $\ell_0^4, \ell_0^2\ell_1^2, \ell_0^2\ell_1\ell_2, \ell_0^2\ell_2^2$
vanish. 

Since the class $\ell_2$ is a pullback from $\mathcal{F}_2$, which is of dimension $2$,
we have $\ell_2^3=0$, implying that $\ell_2^4=\ell_0\ell_2^3=\ell_1\ell_2^3=0$.
Similarly, $\ell_2$ and $\ell_1$ are induced from $\mathcal{F}_1$, which is of dimension
$3$, hence the monomials of degree $4$ in $\ell_1$ and $\ell_2$ vanish.
\end{proof}

Proposition \ref{Chern-Hodge-bundle} together with Lemma \ref{vanishing-monomials}
implies the following relation.

\begin{corollary} \label{lambda3lambda1S}
We have on $\mathcal{F}_0$
$$
\deg(\lambda_3\lambda_1)= \frac{1}{2}(p-1)^4 
\left(\ell_0^3\ell_1+\ell_0^3\ell_2+ \ell_0\ell_1^3+ 3\, \ell_0\ell_1^2\ell_2 + 3\, \ell_0\ell_1\ell_2^2\right)\, .
$$ 
\end{corollary}
Thus we need the intersection numbers defined by the five monomials in $\ell_0, \ell_1, \ell_2$
appearing in Corollary \ref{lambda3lambda1S}.

The intersection numbers 
$(\ell_0\ell_1^3,  \ell_0\ell_1^2\ell_2, \ell_0\ell_1\ell_2^2)$ on $\mathcal{F}_0$
are equal to the intersection numbers 
$(\ell_1^3, \ell_1^2\ell_2, \ell_1\ell_2^2)$ on $\mathcal{F}_1$ as the degree
of $\ell_0$ on a generic fibre of $\pi_0$ is $1$.
\begin{lemma}\label{ell1ell2sq}
We have $\deg \ell_1\ell_2^2= p^2(p^2+1)$ on $\mathcal{F}_1$.
\end{lemma}
\begin{proof}
The space $\mathcal{F}_2$ can be identified with the surface in ${\PP}^3$ over
${\FF}_p$ given by the equation 
$$
x_1 x_4^{p^2}-x_1^{p^2}x_4 + x_2x_3^{p^2}-x_2^{p^2}x_3 =0 
$$
and $\ell_2$ is represented by the pullback under $\pi_1$ of the hyperplane class $h$ on 
$\mathcal{F}_2$. Therefore $h^2$
can be represented by an effective zero cycle of degree $p^2+1$. The surface
$\mathcal{F}_2$ is unirational (see \cite{K2017}), hence $h^2$ can be
represented by $p^2+1$ times a point. The morphism 
$\pi_1$ is inseparable of degree $p$, hence the pullback of
a point $\mathcal{F}_2$ is $p$ times a fibre of $\mathcal{F}_1$.
Since the degree of $\ell_1$ on a fibre of $\pi_1$ is $p$ we get
$\deg(\ell_1\ell_2^2)=p\cdot p \cdot (p^2+1)$. 
\end{proof}
\begin{lemma}\label{relation-lambda4}
We have on $\mathcal{F}_0$ the relation
$$
p\, \ell_0^3\ell_1-(p^2+1)\, \ell_0^3\ell_2 + p\, \ell_0\ell_1^3 - (p-1)^2 \, \ell_0\ell_1^2\ell_2 -
(2p^2-p+2)\ell_0\ell_1\ell_2^2
=0 \, .
$$
\end{lemma}
\begin{proof} This follows from the fact that $\lambda_4$ vanishes
in the Chow ring of $\mathcal{A}_g$ as explained in Section \ref{Space-Ag} and the expression for $\lambda_4$
as a polynomial in the $\ell_i$ by Proposition \ref{Chern-Hodge-bundle} and Corollary 
\ref{Cor-c2Q1}.
\end{proof}
\begin{corollary}\label{relation-lambda4-by ell0}
On $\mathcal{F}_1$ we have the relation
$$
p\, \ell_0^2\ell_1-(p^2+1)\, \ell_0^2\ell_2 + p\, \ell_1^3 - (p-1)^2 \, \ell_1^2\ell_2 -
(2p^2-p+2)\ell_1\ell_2^2
=0 \, .
$$
\end{corollary}
\begin{proof}
We know that $\mathcal{F}_0$ is a ${\PP}^1$-bundle over $\mathcal{F}_1$. Therefore
each cycle class $\xi \in A_k(\mathcal{F}_0)$, the dimension $k$ Chow group of $\mathcal{F}_0$, 
can be written uniquely as $\xi=\pi_{0}^*(\xi_0)+
\pi_{0}^*(\xi_1)\ell_0$ with 
$\xi_0\in A_{k-1}(\mathcal{F}_1)$ and $\xi_1 \in A_k(\mathcal{F}_1)$.
In particular, the map $\xi_1 \mapsto \pi_{0}^*(\xi_1)\ell_0$ is injective. The result thus follows 
from Lemma \ref{relation-lambda4}.
\end{proof}
\begin{lemma}\label{second-relation}
We have on $\mathcal{F}_1$ the relation
$$
2\, \ell_0^2\ell_1-(p-1) \ell_0^2\ell_2 +(p-1)\ell_1^2\ell_2 -2(p^2-p+1) \ell_1\ell_2^2 =0\, .
$$
\end{lemma}
\begin{proof}
We have the exact sequence of Dieudonn\'e modules
$$
0 \to A \to VM_2/pM_2 \to VM_2/VM_1 \to 0\, 
$$
with ${\rm Lie}(Y_2)^\vee = VM_2/pM_2$ and $Q_1 = VM_2/VM_1$.
The total Chern class of the sheaf corresponding to $A$ has the form
$$
c(A)=(1-\ell_2)(1-p\, \ell_2)^{-1}(1+\ell_1+c_2(Q_1))^{-1}\, .
$$
Since ${\rm rank}(A)=2$ the third Chern class should vanish; this gives
a relation on $\mathcal{F}_1$
$$
2\ell_0^2\ell_1-(p-1) \ell_0^2\ell_2 +(p-1)\ell_1^2\ell_2 -2(p^2-p+1) \ell_1\ell_2^2 =0\, .
$$
\end{proof}

\begin{lemma}\label{third-relation}
On $\mathcal{F}_1$ we have the relation
$$
p\, \ell_0^2\ell_2-p\, \ell_1^2\ell_2+2(p^2-p+1)\ell_1\ell_2^2=0 \, .
$$
\end{lemma}
\begin{proof}
Let $H$ be a hyperplane section of $\mathcal{F}_2$ with $H\cap \mathcal{F}({\FF}_{p^2})=\emptyset$.
We work on $\pi_1^{-1}(H)$. Here we have that $\dim M_2/(F,V)M_2=3$ and we thus have a rank~$3$
locally free sheaf $B$ on $H$ determined by $M_2/(F,V)M_2$. Because of the exact sequence
$$
0 \to VM_2/VM_2 \cap FM_2 \to M_2/FM_2 \to M_2/(F,V)M_2 \to 0
$$
we have the exact sequence
$$
0 \to U_2 \to {{\rm Lie}(Y_2)^{(p)}}^{\vee} \to B \to 0\, ,
$$
since $VM_2/VM_2 \cap FM_2 = VM_2/pM_3$. 
We thus find
$$
[B]=[4]-[Q_2^{(p)}]+[Q_2^{(p^2)}]-[U_2]=[4]+[U_2^{(p)}]-[U_2^{(p^2)}]-[U_2]\, .
$$
We also have the inclusions $(F,V)M_2 \subset M_1 \subset M_2$ on $\pi_1^{-1}(H)$ and we thus have
a locally free sheaf $L$ corresponding to $M_1/(F,V)M_2$. In the Grothendieck group we have the
corresponding relation $[B]=[L]+[Q_1^{(p)}]$. Thus we find $[L]=[4]+[U_2^{(p)}]-[U_2^{(p^2)}]-[U_2]-[Q_1^{(p)}]$
and we see that the total Chern class of $L$ is given by
$$
c(L)=\frac{(1-p\ell_2)}{(1-p^2\ell_2)(1-\ell_2)}
\, \frac{1}{(1+p\ell_1+p^2c_2(Q_1))}\, .
$$
But $L$ has rank $1$, so $c_2(L)=0$. With $c_2(Q_1)=(\ell_0^2+\ell_1^2-\ell_2^2)/2$ this gives
$$
(p^4-p^3+\frac{3}{2}p^2-p+1)\ell_2^2-(p^3-p^2+p)\ell_1\ell_2-\frac{1}{2}p^2\ell_0^2
+\frac{1}{2}p^2\ell_1^2=0 \, . 
$$ 
Recall now that the class of $H$ is $\ell_2$. Multiplying the preceding relation by $\ell_2$
and using $\ell_2^3=0$ we find 
$$
p\ell_0^2\ell_2-p\ell_1^2\ell_2+2(p^2-p+1)\ell_1\ell_2^2=0 \, .
$$
\end{proof}
As remarked above we need five intersection numbers:
$$
\ell_0^3\ell_1,\, \ell_0^3\ell_2,\, \ell_0\ell_1^3,\, \ell_0\ell_1^2\ell_2,\, 
\ell_0\ell_1\ell_2^2\, .
$$
We know already the last one by Lemma \ref{ell1ell2sq}.
By multiplying the relations of Lemma \ref{second-relation} and \ref{third-relation} 
 by $\ell_0$ 
we find in total three relations coming from Lemmas 
\ref{relation-lambda4}, \ref{second-relation}
and \ref{third-relation}
 between these five intersection numbers.

\begin{corollary}
We have $\deg(\ell_0^3\ell_1)= p(p^2+1)(p^2-p+1)$.
\end{corollary}
\begin{proof}
The sum of $p$ times the relation of \ref{second-relation} and
$(p-1)$ times that of \ref{third-relation} 
gives the relation 
$2p\, \ell_0^3\ell_1-2(p^2-p+1)\ell_0\ell_1\ell_2^2=0$.
\end{proof}

Using the three relations and
Lemma \ref{ell1ell2sq}  our five intersection numbers 
depend on one unknown.
\begin{corollary}\label{one-unknown}
With $x=\deg(\ell_0\ell_1^2\ell_2)$ we find that 
$$
\deg \left[ 
\begin{matrix} \ell_0^3\ell_1\\ \ell_0^3\ell_2\\ \ell_0\ell_1^3\\ \ell_0\ell_1^2\ell_2\\ \ell_0\ell_1\ell_2^2 \\
\end{matrix} \right] =
 \left[ 
\begin{matrix}
p(p^2+1)(p^2-p+1)\\ x-2p(p^2+1)(p^2-p+1)\\
2(p-1+1/p)x-(p^2+1)^2(2p^2-3p+2)\\ x \\ p^2(p^2+1) \\
\end{matrix} \right] \, .
$$
\end{corollary}
\begin{remark}
We have on $\mathcal{F}_0$
$$
\deg \lambda_1^4 = 8\, (p-1)^4 (p^2+p+1)\left(\frac{\deg(\ell_0\ell_1^2\ell_2)}{p} -(p^2+1)(p-1)^2\right) \, .
$$
Since $\lambda_1$ is ample on $S_4$ this should be positive and this gives
$$
\deg (\ell_0\ell_1^2\ell_2) > p(p^2+1)(p-1)^2 \, .
$$
\end{remark}
We now determine the last intersection number.
Recall that the second Chern class $c_2(Q_1)$ satisfies $c_2(Q_1)=(\ell_0^2+\ell_1^2-\ell_2^2)/2$.
Furthermore, recall the cycle class $[\overline{D(\psi)}]$
of a `horizontal' $a\geq 3$-locus on $\mathcal{F}_1$
given by 
$$
[\overline{D(\psi)}]=p\, \ell_1-(p^2+1)\ell_2+e
$$
with $e$ a class with support in the exceptional fibres as
given in Lemma \ref{Dpsi-class}.
\begin{proposition}
We have $c_2(Q_1)\cdot [\overline{D(\psi)}]=0$ and $c_2(Q_1)\cdot e=0$.
\end{proposition}
\begin{proof}
Since $Q_1$ is the tautological quotient of the
$O_{{\mathcal F}_1}$-module associated to $M_2/FM_2$ 
by the universal rank $2$ subbundle $U_1$, the second Chern class
can be realized as the class of the locus where the fibre of $U_1$ contains a
fixed vector. For this we choose an element $v'$ of $M_2/FM_2$ that has the property
that over each affine part of $\mathcal{F}_2$ with $a_i\neq 0$ (for $i=1,\ldots,4$)
it is of the form
$$
v'=\alpha_5\, v_0+ \alpha_6 \, Fx_2+ \alpha_7 Fx_3+ \alpha_8 Fx_4
$$
with the property that the equation $g_2=0$, that is,
$$
a_1 \, \alpha_8^p-a_1^p\alpha_5^{p-1}\alpha_8 + a_2 \alpha_7^p-a_2^p \alpha_5^{p-1}\alpha_7
+a_3^p\alpha_5^{p-1}\alpha_6-a_3\alpha_6^p=0
$$
has no solutions with $(a_1,a_2,a_3,a_4)\in {\FF}_{p^2}$ with $a_i\neq 0$. Indeed, choosing
$\alpha_5 \neq 0$, $\alpha_6$ and $\alpha_7$ there are only finitely many $\alpha_8$ satisfying
this equation. Then since $\alpha_5\neq 0$,
 we see that this locus has zero intersection with
$\overline{D(\psi)}$. We get $c_2(Q_1)\cdot [\overline{D(\psi)}]=0$. 
By the requirement that we
put over $\mathcal{F}_2({\FF}_{p^2})$ we see that also $c_2(Q_1)\cdot e=0$.
\end{proof}
\begin{corollary}\label{final-stone}
We have $(\ell_0^2+\ell_1^2-\ell_2^2)(p\ell_1-(p^2+1)\ell_2)=0$.
\end{corollary}
\begin{proof}
Recall that  $c_2(Q_1)=(\ell_0^2+\ell_1^2-\ell_2^2)/2$ and
$[\overline{D(\psi)}]=p\ell_1-(p^2+1)\ell_2+e$
with $e$ a class with support in the exceptional fibres.
\end{proof}
By combining Corollary \ref{one-unknown} and Corollary 
\ref{final-stone} we can determine all the intersection numbers.
\begin{corollary}\label{intersection-numbers} We have on $\mathcal{F}_0$
$$
\deg \left[ 
\begin{matrix} \ell_0^3\ell_1\\ \ell_0^3\ell_2\\ \ell_0\ell_1^3\\ \ell_0\ell_1^2\ell_2\\ \ell_0\ell_1\ell_2^2 \\
\end{matrix} \right] =
p\, (p^2+1) \,  \left[ 
\begin{matrix}
p^2-p+1\\ -p^2+p-1\\
-(p-1)^2\\ p^2-p+1 \\ p \\
\end{matrix} \right] \, .
$$
\end{corollary}

Finally we are ready to calculate the coefficient $f_4(p)$ 
of Theorem \ref{2ndThm}.
\begin{theorem}
The class of the supersingular locus 
$S_4 \subset \mathcal{A}_4\otimes {\FF}_{p}$ in the Chow ring
of $\tilde{\mathcal{A}}_4 \otimes {\FF}_p$ equals
$$
[S_4]= (p-1)^3 (p^3-1)(p^4-1)(p^6-1) \lambda_4\lambda_2 \, .
$$
\end{theorem}
\begin{proof}
For each irreducible component $S$ of $S_4$ we calculate the degree of $\lambda_3\lambda_1$ on the model
$\mathcal{F}_0$ of $S$. Indeed, we have $[S_4]=a\lambda_4\lambda_2$ with
$a=\lambda_3\lambda_1 [S_4]/v(4)$ by Proposition \ref{lambda4lambda2}. A calculation
using Corollary \ref{intersection-numbers} and taking into account
the degree $p$ of the map $\mathcal{F}_0 \to S$ (see Lemma \ref{degree-p}) 
yields that $\deg(\lambda_3\lambda_1)$ on $S$ equals $1/p$ times
the degree on $\mathcal{F}_0$ of
$$
\begin{aligned}
(p^2-3p+1)\ell_0^3\ell_1+ (2p^2-2p+2)\ell_0^3\ell_2+ (p^2-3p+1)\ell_0\ell_1^3+&\\ 
4(p-1)^2\ell_0\ell_1^2\ell_2 +(5p^2-7p+5)\ell_0\ell_1\ell_2^2&\\
\end{aligned}
$$
and this equals $(p-1)^4 (p^2+p+1)(p^2+1)$. 
Multiplying this with the number
of irreducible components $(p^2-1)(p^6-1) v(4)$ we find the coefficient
 $a= (p-1)^3(p^3-1)(p^4-1)(p^6-1)$. 
\end{proof}
\end{section}


\begin{thebibliography}{99}
\bibitem{BGM}
P.\ Berthelot, L.\ Breen, W.\ Messing:
{\sl Theorie de Dieudonn\'e Cristalline II.}
Lecture Notes in Mathematics \textbf{930}.
Springer-Verlag, 1982.

\bibitem{Ekedahl1985}
T.\ Ekedahl: {\sl Foliations and inseparable morphisms.} In: 
Algebraic geometry, Bowdoin, 1985, pp. 139–-149. 
Amer.\ Math.\ Soc., Providence, RI, 1987.

\bibitem{Ekedahl1987} T.\ Ekedahl:
{\sl On supersingular curves and abelian varieties.}
Math.\ Scand.\ \textbf{60} (1987), 151--178.

\bibitem{E-vdG} T.\ Ekedahl, G.\ van der Geer:
{\sl Cycle classes of the E-O stratification on the moduli of abelian varieties.}
In: {Algebra, Arithmetic, and Geometry}, Progress in Mathematics \textbf{69}, pp.\
567--636.
Springer Verlag 2009.

\bibitem{E-V} H.\ Esnault, E.\ Viehweg:
{\sl Chern classes of Gauss-Manin bundles of weight $1$ vanish. }
K-theory \textbf{26} (2002), pp.\ 287--305.


\bibitem{F-C} G.\ Faltings, C-L.\ Chai: 
{\sl Degeneration of abelian varieties.}
Ergebnisse der Mathematik \textbf{22}.  Springer Verlag, 1990.

\bibitem{Fulton} W.\ Fulton:
{\sl Intersection Theory.} Second Edition. Springer Verlag 1998.

\bibitem{F-P1689} W.\ Fulton, P.\ Pragacz:
{\sl Schubert varieties and degeneracy loci.}
Lecture Notes in Mathematics 1689.
Springer Verlag, 1998. 

\bibitem{Gan-Hanke-Yu} W.T.\ Gan, J.\ Hanke, J-K.\ Yu:
{\sl On an exact mass formula of Shimura.}
Duke Math.\ Journal
\textbf{107} (2001), pp 103--133.

\bibitem{vdG1} G.\ van der Geer:
{\sl Cycles on the moduli space of abelian varieties.}
In: Moduli of Curves and Abelian Varieties. 65--89.
(The Dutch Intercity Seminar on Moduli), p. 65-89 (Carel Faber and Eduard Looijenga, editors), Aspects of Mathematics, Vieweg, Wiesbaden 1999. 

\bibitem{vdG2} G.\ van der Geer:
{\sl The Chow ring of the moduli space of abelian threefolds.}
J.\ Algebraic Geometry \textbf{7} (1998), 753--770. See
also {\sl Corrigendum} J.\ Algebraic Geometry \textbf{18} (2009) 795--796


\bibitem{Harashita2} S.\ Harashita:
{\sl The $a$-number stratification on the moduli space of 
supersingular abelian varieties.} 
Journal of Pure and Applied Algebra \textbf{193} (2004), 163--191.

\bibitem{H-I} K. Hashimoto, T. Ibukiyama: 
{\sl On class numbers of positive definite binary quaternion hermitian forms.} 
J.\ Fac.\ Sci.\ University Tokyo IA \textbf{27} (1980), 549--601; II. 
\textbf{28} (1981), 695--699; III \textbf{30} (1983), 393--401.

\bibitem{I-K-O} 
T.\ Ibukiyama, T.\ Katsura, F.\ Oort: 
{\sl Supersingular curves of genus two and class numbers.} 
Compositio Math.\ \textbf{57} (1986), 127–-152.

 \bibitem{KYY} V.\ Karemaker, F.\ Yobuko and Chia-Fu Yu: 
{\sl Mass formula and Oort's conjecture for supersingular abelian threefolds.} 
Advances in Math.\ \textbf{386}
(2021), 107812, 52 pp.

\bibitem{K2017}
T.\ Katsura:
{\sl Lefschetz pencils on a certain hypersurface in positive characteristic.}
Higher dimensional algebraic geometry—-
in honour of Professor Yujiro Kawamata's
sixtieth birthday, 265–-278, Adv.\ Stud.\ Pure Math., \textbf{74}, 
Math.\ Soc.\ Japan, 
Tokyo, 2017.

\bibitem{K-O} T.\ Katsura, F.\ Oort: 
{\sl Families of supersingular abelian surfaces.} 
Compositio Math.\ \textbf{62} (1987), 107–-167.

\bibitem{K-O2} 
 T.\ Katsura, F.\ Oort: 
{\sl Supersingular abelian varieties of dimension two or three 
and class numbers.} 
In: Algebraic Geometry, Sendai, 1985. Adv.\ Stud.\ 
Pure Math., vol.\ \textbf{10}, North-Holland, Amsterdam,
1987, pp. 253–-281.



\bibitem{K-Z-L} K-Z.\ Li: 
{\sl Classification of supersingular abelian varieties.}
Math.\ Annalen \textbf{283} (1989), p.\ 331--351.

\bibitem{L-O} K-Z.\ Li, F.\ Oort: 
{\sl Moduli of supersingular abelian varieties.}
Lecture Notes in Mathematics 1680. Springer Verlag 1998.


\bibitem{Norman}
P.\ Norman:
{\sl An algorithm for computing local moduli of Abelian varieties.}
Ann.\ Math.\ (2) {\bf 101}, 499--509 (1975).

\bibitem{Oda-Oort} T.\ Oda, F.\ Oort:
{\sl Supersingular abelian varieties.} 
In: Intern. Symposium on Algebraic Geometry, Kyoto 1977, pp. 595--621.

\bibitem{Oort1974} F.\ Oort: {\sl Subvarieties of moduli spaces.} 
Inventiones Math.\ \textbf{24} (1974), 95--119.

\bibitem{Oort:Stratification} F.\ Oort:
{\sl A stratification of a moduli space
of abelian varieties.} In: Moduli of Abelian varieties (Texel Island).
Editors: C.\ Faber, G.\ van der Geer, F.\ Oort. Progress in Math.
\textbf{195}, Birkh\"auser, Basel, 2001, pp.\ 345--416.

\bibitem{Oort-AM152} F.\ Oort: {\sl Newton polygons and formal groups.}
Annals of Mathematics \textbf{152} (2000), 183--206.

\bibitem{Yu} Chia-Fu Yu:
{\sl The supersingular loci and mass formulas on Siegel modular varieties.}
Documenta Mathematica \textbf{11} (2006), 449--468.

\bibitem{Display}
Th.\ Zink:
{\sl The display of a formal $p$-divisible group.} In:
Cohomologies $p$-adiques et applications arithm\'etiques (I).
Berthelot et al.\ editors. Paris: Soci\'et\'e Math\'ematique de France. Ast\'erisque. {\bf 278}, 127--248 (2002).

\bibitem{Zink} Th.\ Zink:
{\sl On the slope filtration.} Duke Math.\ J.\ \textbf{109} (2001), 79--95.


\end{thebibliography}
\end{document}